\documentclass[12pt]{article}
\usepackage{amssymb,amsbsy,amsmath,amsfonts,amssymb,amscd,amsthm}
\usepackage[T1]{fontenc}
\usepackage[utf8]{inputenc}
\usepackage{graphicx}
\usepackage[all]{xy}
\usepackage{pdfpages}
\usepackage[linktocpage]{hyperref}

\usepackage{latexsym, amssymb, amsthm}
\usepackage{dsfont}
\usepackage{color}
\usepackage{enumerate}

\newcommand{\RR}{\mathds{R}}
\newcommand{\CC}{\mathds{C}}

\newcommand{\EE}{\mathds{E}}
\newcommand{\cL}{\mathcal L}

\DeclareMathAlphabet\gothic{U}{euf}{m}{n}

\newcommand{\cev}[1]{\reflectbox{\ensuremath{\vec{\reflectbox{\ensuremath{#1}}}}}}

\newtheorem{theorem}{Theorem}
\newtheorem{lemma}[theorem]{Lemma}
\newtheorem{proposition}[theorem]{Proposition}
\newtheorem{corollary}[theorem]{Corollary}
\newtheorem{definition}[theorem]{Definition}

\newtheorem{remark}[theorem]{Remark}

\numberwithin{theorem}{section}
\numberwithin{equation}{section}

\date{}
\begin{document} 
%
%\title{An almost  Riesz transform inequality for general manifolds}
\title{A multiplicative inequality of Riesz transform type on general Riemannian manifolds} 
\author{El Maati Ouhabaz\\
Universit\'e de Bordeaux, IMB\\
 351 Cours de la Lib\'eration 33405. Talence, France.\\
Elmaati.Ouhabaz@math.u-bordeaux.fr }

\maketitle     % typeset the header of the contribution

\begin{abstract}
Given any  complete Riemannian manifold $M$, we prove that for every $p \in (1, 2]$ and every $\epsilon > 0$, 
$$ \| \nabla f \|_p^2 \le C_\epsilon \| \Delta^{\frac{1}{2} + \epsilon} f \|_{p}\| \Delta^{\frac{1}{2} - \epsilon} f \|_{p}.$$
The estimate is dimension free. This inequality is even proved  in a general  setting of generators  of some sub-Markov semigroups.\\
We also study Littlewood-Paley-Stein functionals both for Laplace-Beltrami operator on functions and for the Hodge-de Rham Laplacian on differential forms and prove their 
relationship to the boundedness  of the Riesz transform. 
\end{abstract}

\tableofcontents


\section{Introduction}\label{sec1}
Let $M$ be a non-compact complete Riemannian manifold and denote by $\nabla$ and $\Delta$ the corresponding gradient and the non-negative  Laplace-Beltrami operator, respectively. It is a  classical problem in harmonic analysis on manifolds to decide whether  the  Riesz transform $R := \nabla\Delta^{-1/2}$ is bounded on $L^p(M)$. This is obviously true for $p = 2$ by integration by parts. Indeed, 
$$\| \nabla u \|_2^2 =  \int_M \Delta u. u =  \| \Delta^{1/2} u\|_2^2$$
for all $u \in W^{1,2}(M)$ and hence the operator $R$, initially  defined on the range of $\Delta^{1/2}$ (which is dense in $L^2(M)$ because $M$ is non-compact) has a bounded extension to $L^2(M)$. Note that $R$ takes values in $L^2(M, TM)$ where $TM $ is the tangent space. Alternatively, the  Riesz transform may also be defined by ${\bf d}\Delta^{-1/2}$ where ${\bf d}$ is  the differential. In this case $R$ takes values in the $L^2$ space of exact differential forms of order $1$. It is a singular integral operator with a kernel which is not smooth in general. For this reason it is a difficult problem to understand  whether $R$ extends to a bounded operator on $L^p(M)$ for some or all $p \in (1,\infty)$,
$p \not=2$. This problem has been studied  during the last decades. We do not intend to give a complete  account on the subject and we  refer the reader to \cite{Strichartz, ACDH, Bakry, Carron, Chen-Coulhon-Russ, ChenMagOuh, CoulhonDuong, CoulhonCPAM, Guillarmou-Hassell, Hassell-Sikora, Sikora} and the references therein. In particular, 
Bakry \cite{Bakry} proved that if the manifold has non-negative Ricci curvature, then the Riesz transform is bounded on $L^p$ for all $p \in (1, \infty)$. He also proved the boundedness on $L^p$ of the local Riesz transform $\nabla (1 + \Delta)^{-1/2}$ for manifolds with Ricci curvature bounded from below. Coulhon and Duong \cite{CoulhonDuong} proved the boundedness of the Riesz transform on $L^p$ for $p \in (1,2]$ under the assumption that $M$ satisfies the volume doubling property and the heat kernel of $\Delta$ enjoys a Gaussian upper bound. They also gave a counter-example for $p > 2$. The boundedness of the Riesz transform on $L^p$ for $p \in (1,2]$ for general manifolds is a longstanding  open problem. In this general setting Coulhon and Duong \cite{CoulhonCPAM} proved the following multiplicative inequality for $p \in (1,2]$, 
\begin{equation}\label{cd1}
 \| \nabla f \|_p^2   \le C \| \Delta f \|_p \| f \|_p.
 \end{equation}

In this paper  we improve this result  by proving for general  Riemannian  manifold $M$ and  $p \in (1, 2]$,
\begin{equation}\label{aR}
 \| \nabla f \|_p^2 \le C_\epsilon \| \Delta^{\frac{1}{2} + \epsilon} f \|_{p}\| \Delta^{\frac{1}{2} - \epsilon} f \|_{p}
 \end{equation}
for every  $\epsilon \in (0, \frac{1}{2}]$. 
Note that the boundedness of the Riesz transform on $L^p$ corresponds exactly to the limit case $\epsilon = 0$. 

\smallskip
Our strategy of proof is based on Littlewood-Paley-Stein functionals in which we replace the semigroup $e^{-t\Delta}$ by $F(t\Delta)$ for functions $F$ which are holomorphic in a certain  sector and decay at least as $\frac{1}{|z|^\delta}$ at infinity for some $\delta > \frac{1}{2}$. That is, we prove the boundedness on $L^p(M)$ for $p \in (1,2]$ of the functional 
$$ f \mapsto H_F(f) := \left( \int_0^\infty | \nabla F(t \Delta) f |^2\, dt \right)^{1/2}.$$
It is exactly the  decay condition on $F$ which requires  the presence of $\epsilon > 0$ in \eqref{aR}  since we will take  
$F(z) = \frac{1-e^{-z}}{z^{1/2 + \epsilon}}$.  On the other hand, for this function $F$,  the  above Littlewood-Paley-Stein functional cannot be bounded even on $L^2(\RR^d)$
when $\epsilon = 0$. Indeed, if $\epsilon = 0$ then by Fubini and the fact that $\| \nabla u \|_2 = \| \Delta^{\frac{1}{2}} u \|_2$ one has
$$\| H_F(f) \|_2^2 = \int_0^\infty \| \nabla (t\Delta)^{-1/2}(1- e^{-t\Delta}) f \|_2^2\, dt = \int_0^\infty \| (1- e^{-t\Delta}) f \|_2^2\, \frac{dt}{t}$$
which cannot be bounded by $c \|f \|_2^2$. 

The Littlewood-Paley-Stein functionals with holomorphic functions are proved recently by Cometx and Ouhabaz \cite{CometxOuhabaz} for Schr\"odinger operators $\cL = \Delta + V$ on Riemannian manifolds. We reproduce the arguments there to state  and prove several   results in  the general setting  of   generators of  sub-Markovian semigroups.\\
 We  mention  that counter-examples to the boundedness on $L^p$ for $p > 2$ of the classical Littlewood-Paley-Stein functional are discussed in \cite{CometxOuhabaz}. For those examples  \eqref{aR} is  not satisfied for any given $\epsilon$ and any $p > 2$.  The reason is that  \eqref{aR} implies  $\|\nabla e^{-t\Delta} \|_{p\to p} \le \frac{C}{\sqrt{t}}$ 
 (for all $t >0$) but  this latter estimate is not always true for $p > 2$. See \cite{CometxOuhabaz} for more details. If however the manifold satisfies the following gradient estimate 
 $$ | \nabla e^{-t \Delta} f |^2 \le c_\theta\,  e^{-\theta t\Delta} |\nabla f|^2, \quad t >0$$
 for some constants $\theta > 0$ and  $c_\theta > 0$ (possibly large), then we obtain 
  \eqref{aR} for all $p \in [2, \infty)$. This  extends one of the main results in \cite{CoulhonCPAM}. 
 As in \cite{CoulhonCPAM}, 
  the proof appeals to an inequality due to P.A. Meyer \cite{Meyer}, \cite{Meyer2}. We give in an Appendix a short and analytic proof of this inequality in the setting of manifolds which satisfy the above gradient estimate. More precisely we prove that the functional 
$$
 S (f) = \left(  \int_0^\infty e^{-t \alpha \Delta} | \nabla e^{-t\beta \Delta} f |^2\, dt \right)^{1/2}
$$
 is bounded on $L^p(M)$ for all $p \in [2, \infty)$. In Section \ref{sec3}  we give  additional comments on the gradient estimate and discuss how it can be obtained from some integrability of the Ricci curvature of the manifold.  

 \medskip
 Let $ \vec{\Delta} = {\bf d} {\bf d}^* + {\bf d}^* {\bf d}$ be the Hodge-de Rham Laplacian on $1$-differential forms. Here ${\bf d}$ denotes the exterior derivative and ${\bf d}^*$ its formal adjoint. Then $\vec{\Delta}$ is a self-adjoint operator on the $L^2$-space of $1$-forms. We consider the question of boundedness on $L^p$ (i.e., from $L^p$ of forms into $L^p$ of functions) of the Littlewood-Paley-Stein functional
 $$ \cev{H} (\omega) = \left( \int_0^\infty | {\bf d}^* e^{-t\vec{\Delta}} \omega |^2 \, dt \right)^{1/2}.$$
By the standard commutation property  $e^{-t \Delta} {\bf d}^* \omega = {\bf d}^* e^{-t\vec{\Delta}} \omega$ and since
$ ( {\bf d} e^{-t \Delta} )^*= e^{-t \Delta} {\bf d}^*$ 
one can see $\cev{H}$ as a "dual" of the usual Littlewood-Paley-Stein functional  on functions
$$ H (f) = \left( \int_0^\infty | {\bf d} e^{-t \Delta} f |^2 \, dt \right)^{1/2}.$$
  While $H$ is always bounded on $L^p(M)$ for all $p \in (1, 2]$ (cf. \cite{CDL-Studia} or Section \ref{sec2} below), the boundedness of $\cev{H}$ on $L^{p'}$ of $1$-forms is not clear for general manifolds and this question deserves to be studied in details. One of the reasons is its direct  link to the boundedness of the Riesz transform. Indeed, we prove that the boundedness of  $\cev{H}$ is equivalent to the boundedness of the Riesz transform on the dual space. More precisely, we prove the following equivalences for every  $p \in (1, \infty)$ (with conjugate  $p'$).  
%\begin{theo} Let $p \in (1, \infty)$.  %The following assertions are equivalent
\begin{center}
{$a-$ The Riesz transform ${\bf d} \Delta^{-\frac{1}{2}}$ is bounded on $L^p$\\
$\Updownarrow$ \\
$b-$ $ \cev{H}$ is bounded on $L^{p'}$ of $1$-differential forms.\\

\medskip
$c-$ The reverse Riesz inequality $\| \sqrt{\Delta}\, f \|_{p'} \lesssim \| \nabla f \|_{p'}$\\
$\Updownarrow$ \\
$d-$ $\cev{H}$ is bounded on $L^{p'}$ of exact forms (i.e., $\omega = {\bf d} f$).}
\end{center}
From this one deduces in particular boundedness or unboundedness of the functional $ \cev{H}$  on $L^q$ in different 
 settings.

\section{Littlewood-Paley-Stein functionals in an abstract setting}\label{sec2}

Throughout this section, $(X,\mu)$ is a $\sigma$-finite measure space. We consider a non-negative self-adjoint operator $\cL : D(\cL) \subseteq L^2(X) \to L^2(X)$ such that its associated semigroup $(e^{-t\cL})_{t\ge 0}$ is sub-Markovian. In particular, it is a contraction semigroup on $L^p(X)$ for 
$p \in [1, \infty)$. We keep the same notation as on $L^2(X)$ to  denote by $-\cL$ the corresponding generator on $L^p(X)$. 

Next, we consider a linear operator $\Gamma$ acting  on $D(\sqrt{\cL})$ and such that for every $f \in D(\sqrt{\cL})$ and $a.e.\,  x \in X$,  $\Gamma f(x)$ takes values in some  Hilbert space $K_x$. The norm $| . |_x$ of $K_x$ depends smoothly on $x$.  The notation $| \Gamma u |$ stands for 
$| \Gamma u(x) |_x$ at the point $x \in X$. Further, we assume that $\cL$ and $\Gamma$ satisfy the following  properties,
\begin{equation}\label{eq21}
\| \Gamma u \|_2 \le \delta_2 \| \cL^{1/2} u \|_2, \, \, \ u \in D(\cL^{1/2})
\end{equation} 
for some constant  $\delta_2 > 0$.  The $L^2$-norm on the LHS is $\left( \int_X | \Gamma u(x) |_x^2 d\mu(x) \right)^{1/2}$. 
For $p \in (1, 2]$ and   $0 \le f  \in D(\cL) \cap L^\infty(X)$, we assume that $\cL(f^p)$ is a function and the pointwise inequality 
\begin{equation}\label{eq22}
\cL(f^p)(x) \le p f^{p-1}(x) \cL (f)(x) - \alpha_p f^{p-2}(x) | \Gamma (f)(x) |_x^2  \, \, \ a.e.\,  x \in X
\end{equation}
holds for some  $\alpha_p > 0$.  The quantity $f^{p-2} | \Gamma (f) |^2$ is taken to be zero on the set where $f = 0$. We also assume that for $0 \le f  \in D(\cL) \cap L^\infty(X)$,
\begin{equation}\label{eq22-bis}
\int_X \cL(f^p) \, d\mu \ge 0.
\end{equation}
Finally, we also need  the following very natural property
\begin{equation}\label{eq30}
\left[ f_n \to  f\,  a.e.  \ {\rm and} \ \Gamma f_n \to  g \, a.e.  \right]  \Rightarrow g = \Gamma f.
\end{equation}

\smallskip 
%Let us explain the meaning of \eqref{eq22}. 
Let us observe that for $f \in D(\cL)$ we have $f \in D(\sqrt{\cL})$ and hence $|\Gamma (f)| \in L^2$ by \eqref{eq21}. Therefore, the term on the RHS in \eqref{eq22} is a well defined function. Concerning the  term on the LHS, we note that the function $\lambda \mapsto \lambda^p$ is (up to a constant) a normal contraction on $[0,a]$ for a given $a > 0$ and hence $f^p \in D(\sqrt{\cL}) \cap L^\infty$ for $ f  \in D(\cL) \cap L^\infty$. See \cite{Fukushima}, Theorem 1.5.2  or the proof of Theorem 2.25 in \cite{Ouh05}. In particular,  $\cL(f^p)$ is defined at least as an element of the dual space $(D(\sqrt{\cL}))'$ (see e.g.  \cite{Ouh05}, Theorem 1.55) but we assume here that it is a function.  
%The inequality in \eqref{eq22} is interpreted in the usual sense by taking duality with non-negative functions $\varphi \in D(\sqrt{\cL})$. 
Concerning \eqref{eq22-bis}, formally $\int_X \cL(f^p) \, d\mu  = \int_X f^p \cL (1) \, d\mu \ge 0$ because $\cL (1) \ge 0$ by the semi-Markovian assumption on the semigroup.  However in order to make this calculation correct, one needs $f^p$ to belong to the domain of the generator of the semigroup in $L^1$. Assumption \eqref{eq22-bis} is satisfied for many concrete examples.

%Since the semigroup $(e^{-t\cL})_{t\ge0}$ is sub-Markovian, it is a contraction on $L^p(X)$ for all $p \in [1, \infty]$. 
%Therefore, it has a generator which we denote by $-\cL_p$ ($\cL_2 = \cL$). Note that if $0 \le f \in D(\cL_p)$ then $f^p \in D(\sqrt{\cL})$ (see \cite{Ouh05}, Theorem 3.9). %Therefore, $\cL(f^p) $ exists at least as an element of the dual space $(D(\sqrt{\cL}))'$ (see \cite{Ouh05}, Theorem 1.55). Property \eqref{eq22} means in particular that $
%\Gamma(f)$ is defined for all $0 \le f \in D(\cL_p)$.  
We shall apply \eqref{eq22}  and \eqref{eq22-bis} with $u(t) = e^{-t\cL} f$ for $t > 0$ and $0 \le f \in L^2(X) \cap L^\infty(X)$. In this case, $u(t) \ge 0$ and $0 \le u(t)  \in D(\cL) \cap L^\infty$ by positivity of the semigroup, its  analyticity on $L^2$ and the fact that it is a contraction on $L^\infty$. 

We illustrate the above assumptions by the following examples. Let $X = M$ be a  Riemannian manifold and $\cL = \Delta + V$ a Schr\"odinger operator with a non-negative potential  
$ V \in L_{loc}^1(M)$. Set $\Gamma = \nabla$ and  $K_x = T_xM$, the tangent space at $x$. Assumption \eqref{eq21} follows readily from the fact that 
$$ \| \nabla u \|_2^2 \le  \| \nabla u \|_2^2  + \int_M V | f |^2 = \| \cL^{1/2} u \|_2^2.$$
 On the other hand,
\begin{eqnarray*}
\cL(f^p) &=& \Delta(f^p) + V f^p\\
&=& p f^{p-1} \Delta f - p(p-1) f^{p-2} |\nabla f |^2 + Vf^p\\
&=& p f^{p-1} \cL(f) - p(p-1) f^{p-2} |\nabla f |^2 - (p-1) f^{p-2} V f^2\\
&\le & p f^{p-1} \cL(f) - p(p-1) f^{p-2} |\nabla f |^2.
\end{eqnarray*}
This shows that \eqref{eq22} is satisfied for $\alpha_p = p(p-1)$. Assumption \eqref{eq22-bis}  follows from 
$\int_M f^{p-1} \Delta f - (p-1) f^{p-2} |\nabla f |^2  \ge 0$ which can be seen immediately from  
Theorem 3.9 in \cite{Ouh05}.  If we chose $\Gamma = \sqrt{V}$ (i.e., multiplication by $\sqrt{V}$), then the previous calculations give
\begin{eqnarray*}
\cL(f^p) &=& p f^{p-1} \cL(f) - p(p-1) f^{p-2} |\nabla f |^2 - (p-1) f^{p-2} V f^2\\
&\le & p f^{p-1} \cL(f) - (p-1) f^{p-2} |\sqrt{V} f |^2
\end{eqnarray*}
and hence \eqref{eq22} is satisfied for $\alpha_p = p-1$.  \\
Let  $\cL = -{\rm div}(A(x) \nabla)$ be an elliptic operator on $\RR^d$ or a domain and subject to Dirichlet or Neumann boundary conditions. We assume that the coefficients 
are real-valued.  One proves  easily that \eqref{eq22} holds with $\alpha_p = \eta (p-1)$ where $\eta >0$ is the ellipticity constant. The sub-Markovian property \eqref{eq22-bis} can be checked  by the same arguments as previously.  

\smallskip
The following theorem was proved by E.M. Stein for the Euclidean Laplacian (see also \cite{CDL-Studia} for Laplace-Beltrami operator on a  complete Riemannian manifold). The proof given below is known and our main reason to reproduce it here is to convince the reader that the  $L^p$-estimate of the Littlewood-Paley-Stein functional is independent of any kind of dimension.
%%%%%%%%%%%%%%
\begin{theorem}\label{th21} Let $\cL$ be as above and assume \eqref{eq21},  \eqref{eq22}, \eqref{eq22-bis} and \eqref{eq30}. Define 
$$H_\Gamma (f) = \left( \int_0^\infty | \Gamma e^{-t\cL} f |^2 dt \right)^{1/2}.$$
Then  $H_\Gamma$ is bounded on $L^p(X)$ for all $p \in (1,2]$. In addition, there exists a constant $c_p$, depending only on $p$,  such that
\begin{equation}\label{eq23}
\| H_\Gamma (f)  \|_p \le \frac{c_p}{\sqrt{\alpha_p}} \| f \|_p.
\end{equation}
\end{theorem}

\begin{proof}
 By writing $f = f^+ - f^-$ we see that we may assume without loss of generality that $f$ is non negative. Let $0 \le f \in L^1(X) \cap L^2(X)$ and set 
 $u(t) = e^{-t\cL} f$. We shall need $u(t) \to 0$ (in $L^2(X)$) as $t \to \infty$ but this is not true in the case where  $0$ is an eigenvalue of $\cL$. A way to get around this is to replace $\cL$ by $\cL + \epsilon$ for $\epsilon > 0$. It is clear that \eqref{eq21} and \eqref{eq22} hold with the same constants $\delta_2$ and $\alpha_p$ for small $\epsilon > 0$. We might first  prove the theorem for $\cL + \epsilon$ and then let $\epsilon \to 0$. 
 
 By \eqref{eq22},
\begin{equation*}
(\frac{\partial}{\partial t} + \cL) u^p = p u^{p-1}(-\cL u)  + \cL (u^p)
\le -\alpha_p u^{p-2} | \Gamma u |^2.
\end{equation*}
In particular, $(\frac{\partial}{\partial t} + \cL) u^p \le 0$ and  hence 
\begin{eqnarray*}
\int_0^\infty | \Gamma e^{-t\cL}f |^2 dt &\le& \frac{-1}{\alpha_p} \int_0^\infty u(t)^{2-p} (\frac{\partial}{\partial t} + \cL) u(t)^p dt\\
&\le& \frac{1}{\alpha_p} \left( \sup_{t > 0}  e^{-t\cL} f  \right)^{2-p}  \int_0^\infty -(\frac{\partial}{\partial t} + \cL) u^p dt. 
\end{eqnarray*}
%&=& \frac{1}{\alpha_p} \left( \sup_{t > 0}  e^{-tL} f  \right)^{2-p} \left[  u^p(0)  - \int_0^\infty L (u^p) dt \right]\\
%&=& \frac{1}{\alpha_p} \left( \sup_{t > 0}  e^{-tL} f  \right)^{2-p} \left[ f^p  - \int_0^\infty L (u^p) dt \right].
Next, we integrate over $X$ and use H\"older's inequality,
\begin{eqnarray*}
\int_X | H_\Gamma (f) |^p d\mu &\le& \frac{1}{(\alpha_p)^{p/2}} \int_X  \left( \sup_{t > 0}  e^{-t\cL} f  \right)^{(2-p)\frac{p}{2}} 
 \left( -\int_0^\infty (\frac{\partial}{\partial t} + \cL) u^p dt \right)^{p/2}\\
 &\le& \frac{1}{(\alpha_p)^{p/2}} \|  \sup_{t > 0}  e^{-t\cL} f \|_p^{(2-p)\frac{p}{2}} \left( -\int_X  \int_0^\infty (\frac{\partial}{\partial t} + \cL) u^p dt \right)^{p/2}\\
 &=& \frac{1}{(\alpha_p)^{p/2}} \|  \sup_{t > 0}  e^{-t\cL} f \|_p^{(2-p)\frac{p}{2}} \left( \int_X  u^p(0) - \int_0^\infty \int_X \cL(u^p) d\mu  dt \right)^{p/2}.
 %&\le& \frac{1}{(\alpha_p)^{p/2}} \|  \sup_{t > 0}  e^{-t\cL} f \|_p^{(2-p)\frac{p}{2}} \| f \|_p^{p^2/2}.
\end{eqnarray*}
%Note that we use the fact that the semigroup $(e^{-t\cL})_t$ is sub-Markovian to have $\int_X \cL(u^p) d\mu = \int_X u^p \cL(1) d\mu \ge 0$. Thus,
Now we use the sub-Markovian property \eqref{eq22-bis} to end up with  
\begin{equation}\label{eq24}
\| H_\Gamma (f) \|_p \le \frac{1}{\sqrt{\alpha_p}}\, \| \sup_{t > 0}  e^{-t\cL} f \|_p^{\frac{(2-p)}{2}} \| f \|_p^{p/2}.
\end{equation}
It is a classical fact  that the maximal operator $M(f) =  \sup_{t > 0}  | e^{-t\cL} f | $ is bounded on $L^p(X)$ for all $p \in (1,\infty)$. Its norm can be estimated by a constant $\kappa_p$ depending only on $p$, see e.g. \cite{Cowling}, Theorem 7. We  insert this  in \eqref{eq24} to obtain  
%we  obtain the theorem.
\begin{equation}\label{eq24-bis}
\| H_\Gamma (f) \|_p \le \frac{1}{\sqrt{\alpha_p}}\, \| f \|_p
\end{equation}
for all $f \in L^1 \cap L^\infty$. This inequality extends to all $f \in L^p$ by a density argument. Indeed, take a sequence  $f_n \in L^1 \cap L^\infty$ which converges in $L^p$ to $f$. Then by \eqref{eq24-bis}, the sequence $\Gamma e^{-t\cL} f_n$ converges to some $g(t,.)$ in $L^p(X, L^2(0, \infty))$. Next, for any $N > 0$ and $\varphi \in L^{p'}$, it follows immediately from the Cauchy-Schwartz inequality (in $t$) and H\"older's inequality  (for $d\mu$) that 
$\Gamma e^{-t\cL} f_n$ converges to $g(t,.)$ in $L^1(X\times (0,N), d\mu \times dt)$. After extracting a sequence which converges for $a.e.\, (x,t)$ we obtain from assumption \eqref{eq30} that $g(t,.) = \Gamma e^{-t\cL} f (.)$. This gives  \eqref{eq24-bis} for $f\in L^p$. 
\end{proof}

In the rest of this section we borrow a result from the recent article  \cite{CometxOuhabaz} which shows that boundedness of the Littlewood-Paley-Stein functional can be characterized in terms of ${\mathcal R}$-boundedness of a family of operators. 

\begin{definition}\label{def-Rad}
A subset ${\mathcal T}$ of bounded operators on $L^p(X)$ is said  ${\mathcal R}$-bounded (Rademacher bounded) if there exists a constant $C > 0$ such that 
for every collection $T_1,.., T_n \in {\mathcal T}$ and every $f_1,..., f_n \in L^p(X)$
\begin{equation}\label{eq25}
\EE \left\| \sum_{k=1}^n {\gothic r}_k T_k f_k \right\|_p \le C\, \EE \left\| \sum_{k=1}^n {\gothic r}_k f_k \right\|_p.
\end{equation}
Here, $({\gothic r}_k)_k$ is a sequence of independent Rademacher variables and $\EE$ is the usual expectation.
The smallest constant $C > 0$ is called the ${\mathcal R}$-bound of ${\mathcal T}$.
\end{definition}
  By the Kahane-Khintchine inequality, this definition can be reformulated as follows 
\begin{equation}\label{eq26}
\left\|  \left( \sum_{k=1}^n | T_k f_k|^2 \right)^{1/2}  \right\|_p \le C\,  \left\| \left( \sum_{k=1}^n |f_k|^2 \right)^{1/2} \right\|_p.
\end{equation}
The notion of ${\mathcal R}$-bounded family of operators plays a very important role in many questions in functional analysis (see  \cite{HytonenII}) as well as in the theory of maximal regularity for evolution equations (see \cite{Weis} or \cite{KW}). The following theorem was proved in \cite{CometxOuhabaz} in the case of Schr\"odinger operators on manifolds. The proof  remains valid in the general setting we consider here. Some related results were proved  in \cite{LeMerdy} and \cite{Haak} in the context  of admissible operators in control theory. 
%%%%%%%%%%%%%%%%%
\begin{theorem}\label{thm-Rad} Let  $\Gamma $ and $\cL$ satisfy the assumptions of Theorem \ref{th21}. Let  $p \in (1,\infty)$. Then $H_\Gamma$ is bounded on $L^p(X)$ if and only if the set $\{ \sqrt{t}\, \Gamma e^{-t\cL}, t > 0 \}$ is ${\mathcal R}$-bounded on $L^p(X)$. \\
In particular, $\{ \sqrt{t}\, \Gamma e^{-t\cL}, t > 0 \}$ is ${\mathcal R}$-bounded on $L^p(X)$ for all $p \in (1,2]$. The ${\mathcal R}$-bound on $L^p(X)$ depends only on $p$. 
\end{theorem} 
%%%%%%%%%%%%%%%%%%%%%%%
\begin{proof} Let $p \in (1,\infty)$ and 
suppose that  $H_\Gamma$ is bounded on $L^p(X)$.  We prove that the set $\{ \sqrt{t}\, \Gamma e^{-t\cL}, t > 0 \}$ is ${\mathcal R}$-bounded on $L^p(X)$. Let $t_k \in (0, \infty)$ and $f_k  \in L^p(X)$ for $k = 1,...,N$.  Set 
$ I := \EE \left|  \sum_k  {\gothic r}_k \sqrt{t_k}\, \Gamma\,  e^{-t_k \cL} f_k \right|^2$. We denote by $\Gamma(f)(x).\Gamma(g)(x)$ the scalar product of $\Gamma(f)(x)$ and $\Gamma(g)(x)$ in $K_x$. 

If $0$ is an eigenvalue of $\cL$ and $P$ denotes the projection onto $\ker(\cL)$,  then as a consequence of \eqref{eq21},  $\Gamma(P u) = 0$ for all $u \in L^2(X)$. Since $e^{-t\cL}$ converges strongly to $P$ on $L^2(X)$ it then follows that   
$$ I = - \int_0^\infty \frac{d}{dt} \EE | \Gamma e^{-t\cL} \sum_k  {\gothic r}_k \sqrt{t_k} e^{-t_k \cL} f_k |^2\, dt.$$
Using (twice) the independence of the Rademacher variables we have 
\begin{eqnarray*}
I &=& - \int_0^\infty \frac{d}{dt} \EE | \Gamma e^{-t\cL} \sum_k  {\gothic r}_k \sqrt{t_k} e^{-t_k \cL} f_k |^2\, dt\\
&=& 2 \int_0^\infty \EE \left[ (\Gamma e^{-t\cL} \sum_k  {\gothic r}_k \sqrt{t_k}  e^{-t_k \cL} f_k) \cdot (\Gamma e^{-t\cL}\sum_k  {\gothic r}_k \sqrt{t_k} \cL e^{-t_k \cL} f_k)\right] \, dt\\
&=& 2  \int_0^\infty \EE   \sum_k  \Gamma e^{-t\cL} {\gothic r}_k \sqrt{t_k} e^{-t_k \cL} f_k \cdot  \Gamma e^{-t\cL}   {\gothic r}_k \sqrt{t_k} \cL e^{-t_k \cL} f_k \, dt\\
&=& 2 \int_0^\infty \EE  \sum_k  \Gamma e^{-t\cL} {\gothic r}_k e^{-t_k \cL} f_k \cdot  \Gamma e^{-t\cL}   {\gothic r}_k (t_k \cL) e^{-t_k \cL} f_k \, dt\\
&=& 2 \int_0^\infty \EE \left[ (  \Gamma e^{-t\cL} \sum_k {\gothic r}_k e^{-t_k \cL} f_k) \cdot  (\Gamma e^{-t\cL}  \sum_k {\gothic r}_k (t_k \cL) e^{-t_k \cL} f_k )\right] \, dt.
\end{eqnarray*}
Next, by the Cauchy-Schwarz inequality,
\begin{eqnarray*}
I &\le& 2 \int_0^\infty \left( \EE | \Gamma e^{-t\cL} \sum_k {\gothic r}_k e^{-t_k \cL} f_k |^2 \right)^{1/2} \left( \EE | \Gamma e^{-t\cL} \sum_k {\gothic r}_k (t_k \cL) e^{-t_k \cL} f_k |^2 \right)^{1/2}\, dt\\
&\le& \int_0^\infty  \EE | \Gamma e^{-t\cL} \sum_k {\gothic r}_k e^{-t_k \cL} f_k |^2  \, dt + \int_0^\infty  \EE | \Gamma e^{-t\cL} \sum_k {\gothic r}_k (t_k \cL) e^{-t_k \cL} f_k |^2  \, dt.
\end{eqnarray*}
Therefore,
\begin{equation}\label{eq27}
I \le \EE \left[ \left(H_\Gamma ( \sum_k {\gothic r}_k e^{-t_k \cL} f_k)\right)^2 \right]  + \EE \left[ \left(H_\Gamma ( \sum_k {\gothic r}_k (t_k \cL) e^{-t_k \cL} f_k) \right)^2 \right].
\end{equation}
Since 
\[
 \EE \left[ \left(H_\Gamma ( \sum_k {\gothic r}_k e^{-t_k \cL} f_k)\right)^2 \right] = \EE \left\| \sum_k {\gothic r}_k \Gamma e^{-t\cL} e^{-t_k \cL} f_k \right\|^2_{L^2((0,\infty), dt)},
 \]
%\[
%\sqrt{I} \le \left| \EE \| \sum_k {\gothic r}_k \Gamma e^{-tL} e^{-t_k L} f_k \|^2_{L^2((0,\infty), dt)} \right|^{1/2} 
%+ \left| \EE \| \sum_k {\gothic r}_k \Gamma e^{-tL} (t_k L) e^{-t_k L} f_k \|^2_{L^2((0,\infty), dt)} \right|^{1/2}. 
%\]
we have by the Kahane inequality,
\begin{equation}\label{eq28}
c_p \sqrt{I} \le  \left| \EE \left[ \left(H_\Gamma ( \sum_k {\gothic r}_k e^{-t_k \cL} f_k)\right)^p \right] \right|^{1/p}  +  \left| \EE \left[ \left(H_\Gamma ( \sum_k {\gothic r}_k (t_k \cL) e^{-t_k \cL} f_k) \right)^p \right] \right|^{1/p}
\end{equation}
for some constant $c_p > 0$ depending only on $p$. Now we use the assumption that $H_\Gamma$ is bounded on $L^p(X)$ and obtain 
\begin{eqnarray*}
\left\| \sqrt{I} \right\|_p &\le&  C \left( \left| \EE \left\| \sum_k {\gothic r}_k e^{-t_k \cL} f_k \right\|_p^p \right|^{1/p}  + 
 \left| \EE \left\| \sum_k {\gothic r}_k (t_k \cL) e^{-t_k \cL} f_k \right\|_p^p \right|^{1/p} \right)\\
 &\le& C' \left( \EE \left\| \sum_k {\gothic r}_k e^{-t_k \cL} f_k \right\|_p + 
 \EE \left\| \sum_k {\gothic r}_k (t_k \cL) e^{-t_k \cL} f_k \right\|_p \right)
\end{eqnarray*}
where we used again the Kahane inequality.  On the other hand, it is easy to see by the Kahane inequality that 
$\| \sqrt{I} \|_p$ is equivalent to $\EE \left\|  \sum_k  {\gothic r}_k \sqrt{t_k}\,  \Gamma\, e^{-t_k \cL} f_k \right\|_p$.  The constant $C'$ in the previous estimate follows from the norm of $H_\Gamma$ and constants from the Kahane-Khintchine inequality, so it depends only on $p$. 

Since the semigroup $(e^{-tL})$ is sub-Markovian, it is holomorphic on $L^p(X)$ for all $p \in (1,\infty)$ and is a contraction for complex time  (see \cite{Ouh05}, Theorem 3.13). It follows from this, the Cauchy formula and \cite{Weis}, Section 4. d. that $(e^{-t\cL})_{t > 0} $ and $ (t\cL e^{-t\cL})_{t>0} $ are ${\mathcal R}$-bounded on $L^p(M)$ with an ${\mathcal R}$-bound depending only on $p$ (see also \cite{KW} or \cite{HytonenII} Theorem~10.3.4). 
%Since  $-L$ generates a sub-Markovian semigroup, it follows that $L$ has a bounded holomorphic functional calculus on $L^p(X)$ (see 
%\cite{Cowling} and the  optimal result in \cite{Carbonaro}. In the latter reference, $L^p$-estimates of the holomorphic functional calculus are %given in terms of $p$). 
%Therefore, it follows from \cite{KW} or \cite{HytonenII} Theorem~10.3.4 that  $(e^{-tL})_{t > 0} $ and $ (tL e^{-tL})_{t>0} $ are ${\mathcal R}$-bounded on $L^p(M)$.  The ${\mathcal R}$-bound depends only on $p$. 
This  and the previous estimates give
\[
\EE \left\|   \sum_k  {\gothic r}_k \sqrt{t_k}\,  \Gamma\,  e^{-t_k \cL} f_k \right\|_p \le C\, \EE \left\|  \sum_k  {\gothic r}_k  f_k \right\|_p
\]
with a constant $C$ independent of $t_k$ and $f_k$. This proves that $\{ \sqrt{t}\, \Gamma e^{-t\cL}, t > 0 \}$ is ${\mathcal R}$-bounded on $L^p(X)$ and the ${\mathcal R}$-bound depends only on $p$ and does not involve any sort of dimension. 

For the converse, we assume that $\{ \sqrt{t}\, \Gamma e^{-t\cL}, t > 0 \}$ is ${\mathcal R}$-bounded on $L^p(X)$. Let $\delta > \frac{1}{2}$.
% By the Laplace transform, 
%\begin{eqnarray*}
% \sqrt{t}\,  \Gamma\, (I + tL)^{-\delta} &=& c_{\delta} \sqrt{t} \int_0^\infty s^{\delta-1} e^{-s} \Gamma\, e^{-stL} \,ds\\
% &=& c_{\delta} \int_0^\infty a_t(s) \sqrt{s}\,  \Gamma\, e^{-sL} \, ds
 %\end{eqnarray*}
% with $a_t(s) := t^{\frac{1}{2}-\delta} s^{\delta- \frac{3}{2}} e^{-s/t}$. Since $\delta > \frac{1}{2}$, $\int_0^\infty a_t(s) ds = c'_{\delta}$
 %and we can then apply  Lemma 3.2 in  \cite{dePagter} to conclude that $\{ \sqrt{t}\, \Gamma (I + tL)^{-\delta}, t > 0 \}$ is ${\mathcal %R}$-bounded on $L^p(X)$. 
 Set $\cL_\epsilon = \cL + \epsilon $ and $I_\epsilon(f) = \int_0^\infty | \Gamma e^{-t \cL_\epsilon} f |^2\, dt$. Integration by parts and the Cauchy-Schwarz inequality give
 \begin{eqnarray*}
 I_\epsilon(f) &=& \left[  t | \Gamma e^{-t \cL_\epsilon} f |^2 \right]_0^\infty + 2 \int_0^\infty  \Gamma (t \cL_\epsilon) e^{-t \cL_\epsilon} f. \Gamma e^{-t \cL_\epsilon} f\, dt\\
 &\le & 2 \left( \int_0^\infty | \Gamma (t\cL_\epsilon) e^{-t \cL_\epsilon} f |^2\, dt \right)^{1/2} \sqrt{I_\epsilon}.
 \end{eqnarray*}
 This implies
 $$ \sqrt{I_\epsilon(f)}  \le 2 \left( \int_0^\infty | \sqrt{t} \Gamma e^{-\frac{t}{2} \cL}  (t\cL_\epsilon) e^{-\frac{t}{2}\cL_\epsilon} f |^2\, \frac{dt}{t} \right)^{1/2}.$$
 By Lemma \ref{lem1-0} below and the ${\mathcal R}$-boundedness assumption, there exists a constant $C > 0$ depending only on $p$ such that 
 \begin{equation}\label{eq29}
 \| \sqrt{I_\epsilon(f)} \|_p \le C \left\| \left( \int_0^\infty | (t \cL_\epsilon) e^{-\frac{t}{2} \cL_\epsilon} f |^2\, \frac{dt}{t} \right)^{1/2} \right\|_p.
 \end{equation}
 Now, since $e^{-t \cL_\epsilon}$ is sub-Markovian, the operator $\cL_\epsilon$ has a bounded holomorphic functional calculus with a bound depending only on $p$ (cf. \cite{Cowling} or \cite{Carbonaro}). Since $F(\cL_\epsilon) = F_\epsilon(\cL)$ with $F_\epsilon (z) = F(z + \epsilon)$ we see that that the bounds of the holomorphic functional calculus for $\cL_\epsilon$  are independent of $\epsilon$. From this one deduces that  $\cL_\epsilon$ has square function estimates (cf. \cite{CDMY}, Section 6  or \cite{LeMerdy}) with constant independent of $\epsilon$. This and \eqref{eq29} imply
  $$
 \| \sqrt{I_\epsilon(f)} \|_p \le C \| f \|_p $$
 for all $\epsilon > 0$. Then, a  simple  application of Fatou's lemma shows that $H_\Gamma$ is bounded on $L^p(X)$. 
 
 Finally, the assertion concerning $p \in (1,2]$ in the theorem follows from the previous equivalence and Theorem \ref{th21}.
 \end{proof}
 
 \begin{remark} Let $F$ be a bounded holomorphic function on some sector of $\CC^+$ (of angle $w >0$) such that the function $z \mapsto z F(z)$ is also bounded on the same sector. Consider the functionals 
  $H_\Gamma^F$ and $H_\Gamma^{F'}$ as in \eqref{eq32} below. The previous proof shows that if these two functionals are bounded on $L^p$ for some given $p \in (1,\infty)$,  then the set $\{ \sqrt{t}\, \Gamma F(t\cL), \, t > 0 \}$ is ${\mathcal R}$-bounded on $L^p$. 
  \end{remark}
 %%%%%%%
\begin{lemma}\label{lem1-0}
Let $I$ be an interval of $\RR$ and suppose that for each $t \in I$, $S_t$ is a bounded operator on $L^p(X)$.  Then the set $\{ S_t, \ t \in I \}$ is 
${\mathcal R}$-bounded on $L^p(X)$ {\it if and only if} there exists a constant $C > 0$ such that
$$\left\| \left( \int_I  | S_t u(t) |^2 \, dt  \right)^{1/2} \right\|_p \le C\, \left\| \left( \int_I  | u(t) |^2 \, dt  \right)^{1/2} \right\|_p$$
for all $u \in L^p(X, L^2(I))$. The best constant $C$ is equivalent to the  ${\mathcal R}$-bound of $(S_t)_t$. 
\end{lemma}
This lemma is stated in  \cite{Weis} (see 4.a) in the case where  $S_t: L^p(X) \to L^p(X)$.   In the proof of the previous theorem we took $S_t$ to be $ \Gamma e^{-t\cL}$. In this case, $\Gamma e^{-t\cL} f(x) \in K_x$ (again $| \Gamma e^{-t\cL}f(x) |$ is actually $| \Gamma e^{-t\cL}f(x) |_x$). This small change of the context does not affect the statement and proof in  \cite{Weis}. \\
 
 %Suppose now that $ii)$ is satisfied with some $\delta' > \frac{1}{2}$. Define for each $t > 0$, $\phi_t(z) := (1+ tz)^{\delta'}e^{-tz}$. %Then $(\phi_t)_t$ is uniformly bounded in  $H^\infty(\Sigma(\omega_p))$. Hence, $\{ \phi_t(L), \ t > 0 \}$ is 
 %${\mathcal R}$-bounded. Taking the product of the ${\mathcal R}$-bounded operators $\sqrt{t}\, \Gamma\, (1+ tL)^{-\delta'}$ and 
 %+$\phi_t(L)$ gives assertion $i)$. 

We also recall the following simple but useful result from \cite{CometxOuhabaz}.
\begin{proposition}\label{r-bb-resol}
Let $\delta > \frac{1}{2}$ and $p \in (1, \infty)$.  Then  the following assertions are equivalent.\\
i) The set $\{\sqrt{t}\,  \Gamma\, e^{-t \cL}, \ t > 0  \}$ is ${\mathcal R}$-bounded on $L^p(X)$,\\
ii) the set $\{\sqrt{t}\,  \Gamma\, (1+ t \cL)^{-\delta}, \ t > 0  \}$ is ${\mathcal R}$-bounded on $L^p(X)$.\\
The ${\mathcal R}$-bounds in $i)$ and $ii)$ are equivalent.
\end{proposition}
\begin{proof}
We prove $i) \Rightarrow ii)$. By the Laplace transform
\begin{eqnarray*}
 \sqrt{t}\,  \Gamma\, (1+ t \cL)^{-\delta} &=& c_{\delta} \sqrt{t} \int_0^\infty s^{\delta-1} e^{-s} \Gamma\, e^{-st \cL} \,ds\\
 &=& c_{\delta} \int_0^\infty a_t(s) \sqrt{s}\,  \Gamma\, e^{-s\cL} \, ds
 \end{eqnarray*}
 with $a_t(s) := t^{\frac{1}{2}-\delta} s^{\delta- \frac{3}{2}} e^{-s/t}$. Since $\delta > \frac{1}{2}$,  $\int_0^\infty a_t(s) ds = c'_{\delta}$
 and we apply  Lemma 3.2 in  \cite{dePagter} to obtain $ii)$. \\
 For the converse, we define $\phi_t(z) := (1+ tz)^{\delta}e^{-tz}$. Then $(\phi_t)_t$ is uniformly bounded in  $H^\infty(\Sigma(\omega_p))$. Hence, $\{ \phi_t(\cL), \ t > 0 \}$ is 
 ${\mathcal R}$-bounded by \cite{KW} or \cite{HytonenII} Theorem~10.3.4. Taking the product of the ${\mathcal R}$-bounded operators $\sqrt{t}\, \Gamma\, (1+ t\cL)^{-\delta}$ and 
 $\phi_t(\cL)$ gives assertion $i)$. \\
In this argument if one applies \cite{HytonenII} Theorem~10.3.4. assertion 3), then one needs $\cL$ to have a dense range. This later property is not true if $0$ is an eigenvalue of $\cL$. In this case, we replace in the above argument $\phi_t(z)$ by $(1+ t(z+\epsilon))^{\delta}e^{-t(z+\epsilon)}$  (i.e., $\cL$ is replaced by $\cL + \epsilon$) and then let $\epsilon \to 0$ since the bounds are uniform in $\epsilon$.   
\end{proof}

%%%%%%%%%%%%%%%%%%%%%%%
%%%%%%%%%%%%%%%%%%%%%%%
\section{The "almost"  Riesz inequality}\label{sec3}

Let $\Gamma$ and $\cL$ be as in the previous section. Recall that we use the notation $\cL$ for (minus) the generator of the sub-Markovian semigroup $(e^{-t\cL})_{t\ge0}$ on $L^p(X)$, the dependence of $p$ here is implicit.  We  prove that for $p \in (1,2]$ we have the following  "almost" Riesz transform inequality.    

\begin{theorem}\label{thm31}
Suppose that $\Gamma$ and $\cL$ satisfy \eqref{eq21}, \eqref{eq22}, \eqref{eq22-bis}  and \eqref{eq30} and let $p_0, p_1, p \in (1, 2]$ such that $\frac{1}{p_0} + \frac{1}{p_1} = \frac{2}{p}$.  Then for any $\epsilon \in (0, \frac{1}{2}]$, there exists a constant $C_\epsilon = C_\epsilon(p, p_0, p_1)$ such that
\begin{equation}\label{eq31}
\| \Gamma f \|_p^2 \le C_\epsilon \| \cL^{\frac{1}{2} + \epsilon} f \|_{p_0}\| \cL^{\frac{1}{2} - \epsilon} f \|_{p_1}
\end{equation}
for all $f \in D(\cL^{\frac{1}{2} + \epsilon}) \cap D(\cL^{\frac{1}{2} - \epsilon})$.
\end{theorem}

\begin{remark}\label{rem32}
1- We can of course take $p_0= p_1 = p$ and obtain 
$$\| \Gamma f \|_p^2 \le C_\epsilon \| \cL^{\frac{1}{2} + \epsilon} f \|_{p}\| \cL^{\frac{1}{2} - \epsilon} f \|_{p}.$$
The  boundedness of the Riesz transform type operator $\Gamma {\mathcal L}^{-1/2}$ on $L^p$ is exactly the case $\epsilon = 0$. \\
2- As we already mentioned in the introduction, for the  Laplacian on a complete Riemannian manifold, the above result was proved  in \cite{CDL-Studia} (see also \cite{CoulhonCPAM}) in  the particular case $\epsilon = \frac{1}{2}$ which reads as (take $p_0 = p_1 = p$) 
$$ \| \nabla f \|_p^2 \le C \| \Delta f \|_{p}\| f \|_{p}.$$
2- In the case where $\Gamma = \nabla$, $\cL= \Delta$ on a Riemannian manifold  $M$, the estimate \eqref{eq31} is dimension-free. \\
3- We make precise the meaning of \eqref{eq31}. If we denote by $-\cL_r$ the  generator of the semigroup $(e^{-t\cL})$ on $L^r(X)$, then \eqref{eq31} holds for all $ f \in D(\cL_{p_0}^{\frac{1}{2} + \epsilon}) \cap D(\cL_{p_1}^{\frac{1}{2} - \epsilon})$. 
\end{remark} 

Theorem \ref{thm31} applies to Schr\"odinger operators $\cL = \Delta + V$ with non-negative potential $V \in L^1_{loc}(M)$ and $M$ is  any complete Riemannian manifold. It applies with $\Gamma = \nabla$ as well as with $\Gamma = \sqrt{V}$. This gives for $p \in (1, 2]$,
$$ \| \nabla f \|_p^2  + \| \sqrt{V} f \|_p ^2 \le C_\epsilon \| \cL^{\frac{1}{2} + \epsilon} f \|_{p}\| \cL^{\frac{1}{2} - \epsilon} f \|_{p}.$$
The theorem also applies to   $\cL = -{\rm div}(A(x) \nabla)$  an elliptic operator on $\RR^d$ or on any domain of $\RR^d$ and subject to Dirichlet or Neumann boundary conditions. We assume the coefficients to be real-valued but not necessarily bounded. If the coefficients are real and bounded then the corresponding heat kernel has a Gaussian bound (a little of regularity of the domain is needed if the operator is subject to Neumann boundary conditions). In this case, the Riesz transform is bounded on $L^p$ for $p \in (1,2]$. Here, for \eqref{eq31} we do not need boundedness of the coefficients or  regularity of the domain.
\medskip 

The proof of the previous theorem is based on the next proposition which shows the boundedness of a Littlewood-Paley-Stein functional associated with a general function $F$,
\begin{equation}\label{eq32}
H_\Gamma^F (f) = \left( \int_0^\infty | \Gamma F(t\cL) f |^2\, dt \right)^{1/2}.
\end{equation}
In the case of manifolds a more general version involving a family of function $(F_n)_n$ is proved in \cite{CometxOuhabaz}. The arguments there can be extended to the  setting considered here.

Recall the classical Littlewood-Paley-Stein functional
$$H_\Gamma (f) = \left( \int_0^\infty | \Gamma e^{-t\cL} f |^2\, dt \right)^{1/2}$$
with $\Gamma = \nabla$ in the setting of Riemannian manifolds. We consider this functional in an abstract setting and with a more  general family of operators $F(t\cL)$ at the place of  the semigroup 
$(e^{-t\cL})$.   For a given $w \in (0, \pi)$, we denote by $\Sigma(w)$ the open sector
$\{ z \not = 0, \ | arg(z) | < w \}$. By $H^\infty(w)$ we denote the space of bounded holomorphic functions on $\Sigma(w)$. 
%It is endowed with the norm $\| f \|_{w,\infty} = \sup \{ | f(z) |, \ z \in \overline{\Sigma(w)} \}$.   

\begin{proposition}\label{pro33} Suppose that $\Gamma$ and $\cL$ satisfy \eqref{eq21}, \eqref{eq22}, \eqref{eq22-bis} and \eqref{eq30}. Given $p \in (1,2]$ and let $w_p \in (\arcsin | \frac{2}{p} -1|, \pi)$. Let $F \in H^\infty(w_p)$ such that 
\begin{itemize}
\item there exists $\delta > \frac{1}{2}$ such that $ |F(z) | \le \frac{C}{|z|^\delta}$ as $|z| \to \infty,\  z \in \Sigma(w_p)$, 
\item there exists $\epsilon > 0$ such that $| F'(z) | \le \frac{C}{|z|^{1-\epsilon}}$ as $z \to 0, \ z \in \Sigma(w_p).$  
\end{itemize}
Then $H_\Gamma^F$ is bounded on $L^p(X)$, i.e., there exists a constant $C_F = C_F(\epsilon, p, \delta)$ such that
$$ \| H_\Gamma^F(f) \|_p \le C_F \| f \|_p$$
for all $f \in L^p(X)$. If $X = M$ is a complete Riemannian manifold, then $C_F$ is independent of the dimension.\\
If $p > 2$, then the same result holds provided  $H_\Gamma$ is bounded on $L^p(X)$. 
\end{proposition} 
\begin{proof} Let $ F \in H^\infty(w_p)$ having a decay $ |F(z) | \le \frac{C}{|z|^\delta}$ at infinity and suppose first that 
$|F(z)| \le C | z |^{\epsilon}$ for some $\epsilon > 0$. Set $G(z) = (1+z)^{\delta'} F(z)$ for some $\delta' \in (\frac{1}{2}, \delta)$. Then 
$G(z)$ has a decay $C |z|^\epsilon$ at $0$ and $\frac{C}{|z|^{\delta -\delta'}}$ at $\infty$. By Theorem \ref{thm-Rad} and Proposition \ref{r-bb-resol} the set $\{ \sqrt{t} (1+t\cL)^{-\delta'}, \ t > 0 \}$ is ${\mathcal R}$-bounded on $L^p(X)$. 
Now we write
$$H_\Gamma^F (f) =  \left( \int_0^\infty | \sqrt{t} \Gamma (1+ t\cL)^{-\delta'} G(t\cL) f |^2\, \frac{dt}{t} \right)^{1/2}$$
and apply Lemma \ref{lem1-0} to obtain
$$ \| H_\Gamma^F(f) \|_p \le C \left\| \left( \int_0^\infty | G(t\cL) f |^2\, \frac{dt}{t} \right)^{1/2} \right\|_p.$$
The constant $C$ depends on $p$ and $\delta'$. The term on the right hand side is a square function of the operator $\cL$. Since $\cL$ has a bounded holomorphic functional calculus on $L^p(X)$ on the sector $\Sigma(w_p)$ (cf. \cite{Carbonaro}), it follows that $\cL$ has a square function estimate (cf. \cite{CDMY}, Section 6). This means that the latest term is bounded by $C' \| f \|_p$ and we obtain the claim of the proposition.

Suppose now that $F$ does not have decay at $0$ but  $| F'(z) | \le \frac{C}{|z|^{1-\epsilon}}$. We wish to make an integration by parts in the definition of $H_\Gamma^F$  but there are some problems in the case where $0$ is an eigenvalue of $\cL$. In order to get around this we first prove the result for $\cL_\nu = \cL + \nu I$ for $\nu > 0$ in place of $\cL$. We denote by $H_\Gamma^{F,\nu}$ the corresponding  Littlewood-Paley-Stein functional.\\
 It is not difficult to see that the derivative with respect to $t$ of $F(t\cL_\nu)f$ is $\cL F'(t\cL_\nu)f$ and $ t | \Gamma F(t\cL_\nu) f |^2 $ tends to  $0$ when $t \to \infty$ (see also \cite{CometxOuhabaz}, proof of Theorem 4.1).  Therefore, 
\begin{eqnarray*}
(H_\Gamma^{F,\nu}(f))^2 &=& \left[ t | \Gamma F(t\cL_\nu) f |^2 \right]_0^\infty - 2\int_0^\infty \Gamma (t\cL_\nu) F'(t\cL_\nu)f.\Gamma F(t\cL_\nu)f\, dt\\
&=& - 2\int_0^\infty \Gamma (t\cL_\nu) F'(t\cL_\nu)f.\Gamma F(t\cL_\nu)f\, dt\\
&\le& 2 \left( \int_0^\infty | \Gamma (t\cL_\nu) F'(t\cL_\nu)f |^2 \, dt\right)^{1/2} H_\Gamma^{F,\nu}(f).
\end{eqnarray*}
From this we obtain 
$$ H_\Gamma^{F,\nu}(f) \le 2 \left( \int_0^\infty | \Gamma (t\cL_\nu) F'(t\cL_\nu)f |^2 \, dt\right)^{1/2} = 2 H_\Gamma^{\tilde{G},\nu}(f)$$
with $\tilde{G}(z) = z F'(z)$. By the Cauchy formula, $G$ has the same decay as $F$ at $\infty$ and by our assumption it has a decay $C | z |^\epsilon$ at $0$. Then we argue exactly as in the previous case with $\tilde{G}$ in place of $F$. Since the semigroup $(e^{-t\cL_\nu})$ is sub-Markovian we obtain 
$$ \| H_\Gamma^{\tilde{G},\nu}(f) \|_p \le C' \| f \|_p$$
with some constant $C'$ independent of $\nu$ (and any kind of dimension). So
\begin{equation}\label{eq33}
\left\| \left( \int_0^\infty | \Gamma F(t(\cL + \nu)) f |^2\, dt \right)^{1/2} \right \|_p \le 2 C' \| f \|_p.
\end{equation}
For $f \in L^2(X) \cap L^p(X)$ one has $F(t(\cL + \nu))f \to F(t\cL)f$ in $L^2(X)$ as $\nu \to 0$ (at least for each fixed $t \ge 0$) and it follows from \eqref{eq21} that $\Gamma F(t(\cL+\nu))f$ converges to $\Gamma F(t\cL)f$. We apply Fatou's lemma in \eqref{eq33} and a simple density argument to obtain that $H_\Gamma^F$ is bounded on $L^p(X)$.
\end{proof}
%%%%%%%%
\begin{proof}[Proof of Theorem \ref{thm31}]
For $f \in D(\cL^{1/2})$  we have
\begin{equation}\label{eq34}
\frac{1}{2} | \Gamma f|^2 = \int_0^\infty \Gamma (1- e^{-t\cL}) f. \Gamma \cL e^{-t\cL} f\, dt.
\end{equation}
To see this, let us denote by $P$ the projection onto $\ker(\cL)$ if $0$ is an eigenvalue of $\cL$. Then $e^{-t\cL}f$ converges to $Pf$  as $t \to \infty$. As in the proof of Proposition \ref{pro33}, $\Gamma e^{-t \cL} f$  converges to 
 $\Gamma P(f)$, which is $0$ by \eqref{eq21}. Now, the right hand side of \eqref{eq34} is
 $$ \left[ \frac{1}{2} | \Gamma (1-e^{-t\cL})f |^2 \right]_0^\infty = \frac{1}{2} | \Gamma f |^2.$$
 In order to continue we denote, as  in Remark \ref{rem32}, by $-\cL_r$ the generator of $(e^{-t\cL})$ on $L^r(X)$ with again $\cL = \cL_2$. Let  $\epsilon \in (0, \frac{1}{2})$ and 
 consider the functions   $\varphi(z) = \frac{ 1- e^{-z}}{ z^{1/2 +\epsilon}}$ and $\psi(z) = z^{1/2 + \epsilon}e^{-z}$. Let 
 $f \in D(\cL) \cap D(\cL_{p_0}^{\frac{1}{2} + \epsilon}) \cap D(\cL_{p_1}^{\frac{1}{2} - \epsilon})$. We have by  \eqref{eq34}
 \begin{eqnarray*}
 \frac{1}{2} | \Gamma f|^2 &=& \int_0^\infty \Gamma (1- e^{-t\cL})(t\cL)^{-\frac{1}{2} - \epsilon} (\cL^{\frac{1}{2} + \epsilon}f). \Gamma 
 (t\cL)^{\frac{1}{2} + \epsilon} e^{-t\cL} (\cL^{\frac{1}{2} - \epsilon}f)\, dt\\
 &=& \int_0^\infty \Gamma \varphi(t\cL) (\cL^{\frac{1}{2} + \epsilon}f). \Gamma 
 \psi(t\cL) (\cL^{\frac{1}{2} - \epsilon}f)\, dt\\
 &\le& \left( \int_0^\infty | \Gamma \varphi(t\cL) (\cL^{\frac{1}{2} + \epsilon}f) |^2\, dt \right)^{1/2} \left( \int_0^\infty | \Gamma \psi(t\cL) (\cL^{\frac{1}{2} - \epsilon}f) |^2\, dt \right)^{1/2} \\
 &=& H_\Gamma^\varphi(\cL^{\frac{1}{2} + \epsilon}f) H_\Gamma^\psi(\cL^{\frac{1}{2} - \epsilon}f).
 \end{eqnarray*}
 The functions $\varphi$ and $\psi$ satisfy the assumptions of Proposition~\ref{pro33} and hence $H_\Gamma^\varphi$ and $H_\Gamma^\psi$ are bounded on $L^{p_0}$ and $L^{p_1}$ for $p_0, p_1 \in (1, 2]$.   Their norms are bounded by  constants depending only on $\varphi, \psi, p_0, p_1$ and $\epsilon$. From the previous calculations and a simple application of H\"older's inequality we obtain
 \begin{eqnarray*}
  \frac{1}{2}\| \Gamma f \|_p^p &\le&  \| H_\Gamma^\varphi(\cL^{\frac{1}{2} + \epsilon}f) \|_{p_0}^{\frac{p}{2}} 
  \| H_\Gamma^\psi(\cL^{\frac{1}{2} - \epsilon}f) \|_{p_1}^{\frac{p}{2}} \\
  &\le& C \| \cL^{\frac{1}{2} + \epsilon}f \|_{p_0}^{\frac{p}{2}} \| \cL^{\frac{1}{2} - \epsilon}f \|_{p_1}^{\frac{p}{2}}.
  \end{eqnarray*}
  This proves \eqref{eq31} for $\epsilon \in (0, \frac{1}{2})$ and $f \in D(\cL) \cap D(\cL_{p_0}^{\frac{1}{2} + \epsilon}) \cap D(\cL_{p_1}^{\frac{1}{2} - \epsilon})$. For $\epsilon = \frac{1}{2}$ we use the formula
  $$ \frac{1}{2} | \Gamma f|^2 = \int_0^\infty \Gamma e^{-t\cL} f. \Gamma \cL e^{-t\cL} f\, dt$$
  instead of \eqref{eq34} and argue as above. \\
  What remains  to do is to extend the estimate to all $f \in D(\cL_{p_0}^{\frac{1}{2} + \epsilon}) \cap D(\cL_{p_1}^{\frac{1}{2} - \epsilon})$. 
 We start with   $f \in L^2(X) \cap L^{p_0}(X) \cap L^{p_1}(X)$ and $t > 0$. By analyticity of the semigroup we have $e^{-t\cL} f \in 
  D(\cL) \cap D(\cL_{p_0}^{\frac{1}{2} + \epsilon}) \cap D(\cL_{p_1}^{\frac{1}{2} - \epsilon})$. Hence 
  \begin{equation}\label{eq36}
   \| \Gamma e^{-t\cL} f \|_p^2  \le C \| \cL_{p_0}^{\frac{1}{2} + \epsilon}e^{-t\cL}f \|_{p_0} \| \cL_{p_1}^{\frac{1}{2} - \epsilon}e^{-t\cL}f \|_{p_1}.
   \end{equation}
 This latter estimate extends easily to $f \in L^{p_0}(X) \cap L^{p_1}(X)$.  Indeed, for such $f$ we take a sequence 
 $(f_n) \in L^2(X) \cap L^{p_0}(X) \cap L^{p_1}(X)$ which converges to $f$ in $L^{p_0}(X)$ and $ L^{p_1}(X)$. Then $\Gamma e^{-t\cL}f_n$ is bounded in $L^p(X)$ and one extracts a sub-sequence which converges weakly in $L^p(X)$ and then use \eqref{eq30}. Finally, for $f \in D(\cL_{p_0}^{\frac{1}{2} + \epsilon}) \cap D(\cL_{p_1}^{\frac{1}{2} - \epsilon})$ we obtain immediately from   \eqref{eq36} (once we commute powers of $\cL$ and the semigroup)
 $$  \| \Gamma e^{-t\cL} f \|_p^2  \le C \| \cL_{p_0}^{\frac{1}{2} + \epsilon} f \|_{p_0} \| \cL_{p_1}^{\frac{1}{2} - \epsilon} f \|_{p_1}.$$
 Since the right hand side is independent of $t > 0$ we argue exactly as before and use \eqref{eq30} to let $t \to 0$. This finishes the proof of the theorem.
\end{proof}
 
 The constant $C$ depends on $\epsilon$ and this does not allow to let  $\epsilon \to 0$ to obtain the boundedness of the Riesz transform on $L^p(X)$. 
 
 \medskip
We mention the following corollary.
\begin{corollary}\label{cor33}
Suppose the assumptions of Theorem \ref{thm31}. Let $\alpha \in [0,\frac{1}{2})$. Then $\Gamma \cL^{-\alpha} e^{-\cL}$ is bounded on $L^p$ for all $p \in (1,2]$.
\end{corollary}
\begin{proof}
Let $\epsilon > 0$ such that $\frac{1}{2} -\alpha -\epsilon \ge 0$. By Theorem \ref{thm31}
$$\| \Gamma \cL^{-\alpha} e^{-\cL} f \|_p^2 \le C \| \cL^{\frac{1}{2} -\alpha + \epsilon} e^{-\cL} f \|_{p}\| \cL^{\frac{1}{2} -\alpha - \epsilon} e^{-\cL}f \|_{p}.$$
The holomorphy of the semigroup $(e^{-t\cL})$ on $L^p$ implies that the operators $\cL^{\frac{1}{2} -\alpha + \epsilon} e^{-\cL}$ and $\cL^{\frac{1}{2} -\alpha - \epsilon} e^{-\cL}$
are bounded on $L^p$. This gives the corollary. \end{proof} 

This corollary was  proved in \cite{Li} in the setting of Riemannian manifolds.  See also \cite{Devyver} for related results on some fractal-like cable systems. 
 
 %%%%%%%%%%%%%%%%%%%%%%%%%%%%%%%%%%%%%%%%%%%%
 \section{The case $p > 2$}\label{sec5}
 
The sole reason to take $p$ in $ (1,2]$ in Theorem \ref{thm31} is  to use the  boundedness of the Littlewood-Paley-Stein functional $H_\Gamma$ on $L^p(X)$ (cf. Theorem \ref{th21}). If $p > 2$ and if one knows  that $H_\Gamma$ is bounded on $L^p(X)$, then \eqref{eq31} holds by the same proof. 

Suppose now that $X = M$ is a complete non-compact Riemannian manifold for which  the following inequality  holds for some constants $ \theta > 0$ and $c_{\theta}  > 0$, 
\begin{equation}\label{eq38}
| \nabla e^{-t \Delta} f |^2  \le c_{\theta} \,  e^{-\theta t\Delta} |\nabla  f|^2, \quad t >0.
\end{equation}
Since the semigroup $(e^{-t\Delta})$ is sub-Markovian, this inequality holds if 
$$ | \nabla e^{-t \Delta} f | \le \sqrt{c_\theta}\,  e^{-\theta t\Delta} |\nabla f|, \quad t >0.$$
The gradient estimate  \eqref{eq38}  is discussed in \cite{CoulhonCPAM} and  \cite{HQLi} (p. 373). Also
 \cite{ACDH} (Lemma 3.3) gives some conditions under which \eqref{eq38} holds.  
In particular, \eqref{eq38} is more general than assuming the manifold to have non-negative Ricci curvature, see  
Proposition \ref{propKato}  and the discussion below. 

\begin{proposition}\label{prop3-333}
Suppose \eqref{eq38}. Let $p \in [2, \infty)$ and 
let $p_0, p_1 \in (1, \infty)$ such that $\frac{1}{p_0} + \frac{1}{p_1} = \frac{2}{p}$. Then for any $\epsilon \in (0, \frac{1}{2}]$, there exists a constant $C_\epsilon = C_\epsilon(p, p_0, p_1)$ such that
\begin{equation}
\| \nabla f \|_p^2 \le C_\epsilon \| \Delta^{\frac{1}{2} + \epsilon} f \|_{p_0}\| \Delta^{\frac{1}{2} - \epsilon} f \|_{p_1}
\end{equation}
for all $f \in D(\Delta^{\frac{1}{2} + \epsilon}) \cap D(\Delta^{\frac{1}{2} - \epsilon})$.
\end{proposition} 

Proposition \ref{prop3-333}  extends one of the main results in \cite{CoulhonCPAM} (cf. Theorem 4.1). More precisely, it is proved there that if the gradient estimate \eqref{eq38} (with $\theta = 1$) is satisfied, then the multiplicative inequality 
$$\| \nabla f \|_p^2 \le C \|  f \|_p\| \Delta f \|_p$$
holds for $p \in [2, \infty)$. 

\begin{proof} As we already mentioned before, the above result can be proved as for  Theorem \ref{thm31} as soon as we prove that the Littlewood-Paley-Stein functional 
$H_\nabla$ is bounded on $L^p(M)$ for all $p \in [2, \infty)$. To do this, we write $\nabla e^{-t\Delta} = \nabla e^{-\frac{t}{2}\Delta} e^{-\frac{t}{2}\Delta}$ and hence by \eqref{eq38}
%Do do this, take $\delta  = \frac{1}{1+ \theta}$ and write by the semigroup property $e^{-t\Delta} = e^{- \delta t\Delta} e^{-(1-\delta)t\Delta}$. 
%Then by \eqref{eq38} and the definition of $\delta$,
%\begin{eqnarray*}
  $$\int_0^\infty | \nabla e^{-t\Delta} f |^2\, dt \le  c_{\theta}  \int_0^\infty  e^{-\frac{\theta t}{2}\Delta} | \nabla e^{-\frac{t}{2} \Delta} f |^2\, dt.$$
 % &\le& c_\theta^2 \int_0^\infty  e^{-\theta \delta t \cL} | \Gamma e^{-(1-\delta)t\cL}  f |^2\, dt\\
 % &=&  c_\theta \int_0^\infty  e^{-\theta \delta t \Delta} | \nabla e^{-\theta \delta t\Delta}  f |^2\, dt\\
   %&=&  \frac{c_\theta}{\theta \delta}  \int_0^\infty  e^{-t \Delta} | \nabla e^{- t\Delta}  f |^2\, dt.
 % \end{eqnarray*}
By P.A. Meyer's inequality of the Appendix below, % \cite{Meyer}\footnote{see also the Appendix  below}, 
the functional $S$  defined by 
%\begin{equation}\label{S-def}
 $$S(f) = \left( \int_0^\infty  e^{-\frac{\theta t}{2}\Delta} | \nabla e^{-\frac{t}{2} \Delta} f |^2\, dt \right)^{1/2}$$
% \end{equation} 
is bounded on $L^p(M)$ for all $p \in [2, \infty)$. Therefore, $H_\nabla$ is bounded on $L^p(M)$ for $p \in [2, \infty)$.  
\end{proof}

Now we discuss a condition under which the gradient estimate \eqref{eq38} is satisfied. We define  the Hodge-de Rham Laplacian 
$$ \vec{\Delta} := {\bf d} {\bf d}^* + {\bf d}^* {\bf d}$$
on $L^2$ of $1$-differential forms over $M$. Here ${\bf d}$ denotes the exterior derivative and ${\bf d}^*$ its adjoint. It is a classical fact that 
$\vec{\Delta}$ is a self-adjoint operator with an appropriate domain. The Weitzenb\"ock  formula states that 
$$ \vec{\Delta} = \tilde{\Delta} + {\mathcal R}$$
where $\tilde{\Delta} = \nabla^* \nabla$ is the rough Laplacian ($\nabla$ is the connexion) and ${\mathcal R}$ is the Ricci tensor. We denote by 
$V(x)$ the smallest eigenvalue of ${\mathcal R}(x)$  and set $W(x) = \max(W_1(x), 0)$ where $W_1(x)$ is the largest  eigenvalue of 
${\mathcal R}(x)$. The well known domination property states that 
\begin{equation}\label{eq-domination} 
| e^{-t \vec{\Delta}} \omega(x) | \le e^{-t(\Delta + W + V)} |\omega(x)|
\end{equation}
for all $t > 0$ and for all  $1$-form $\omega$ in $L^2$. Note that if $M$  has non-negative Ricci curvature then $V \ge 0$ and the above domination shows that the semigroup on forms is dominated by the semigroup $e^{-t \Delta}$ on functions. Since $V$ is in general non-positive one needs conditions to guarantee that the Schr\"odinger semigroup $e^{-t(\Delta + W + V)}$ exists. This is the case at least on $L^2(M)$  if  $|V|$ is $\Delta$-bounded in the quadratic form sense with bound $< 1$. 

\begin{proposition}\label{propKato} 
Suppose that the semigroups $(e^{-t(\Delta + W + V)})$ and $(e^{-t(\Delta + 2(W + V))})$ exist on $L^2(M)$ and that $(e^{-t(\Delta + 2(W + V))})$ is uniformly bounded on $L^\infty(M)$. Then the gradient estimate \eqref{eq38} holds with $\theta = 1$.
\end{proposition}
\begin{proof} Let  $h \in L^2(M)$. By the Feynman-Kac formula and the Cauchy-Schwarz inequality 
\begin{eqnarray*}
e^{-t(\Delta + W + V)} h (x) &=& \EE_x \left[ e^{-\int_0^t (W+V)(X_s)\, ds} h(X_t) \right]\\
 &\le& \left( \EE_x \left[ e^{-\int_0^t 2(W+V)(X_s)\, ds} \right] \right)^{1/2} \left( \EE_x h^2(X_t) \right)^{1/2}\\
 &=& \left( e^{-t(\Delta + 2 (W + V))} 1 \right)^{1/2} \left( e^{-t\Delta} h^2 \right)^{1/2}.
 \end{eqnarray*}
 Therefore, if $(e^{-t(\Delta + 2(W + V))})$ is uniformly bounded on $L^\infty(M)$, there exists a constant $c > 0$ such that for all $t > 0$
 $$ | e^{-t(\Delta + 2 (W + V))} 1 | \le c.$$
 It follows that 
 $$ | e^{-t(\Delta + W + V)} h (x) |^2 \le c\,  e^{-t\Delta} h(x)^2.$$
 This and the  domination property \eqref{eq-domination} applied with $ \omega = {\bf d} f$ give 
 $$ | e^{-t \vec{\Delta}} {\bf d} f |^2 \le c\, e^{-t\Delta} | {\bf d} f |^2.$$
 Finally, by the commutation property  $e^{-t\vec{\Delta}} {\bf d} = {\bf d} e^{-t \Delta}$, we obtain the proposition. 
\end{proof} 

Next, suppose that there exist $\alpha > 0$ and $\gamma \in [0, 1)$ such that 
\begin{equation}\label{eq-Miyadera}
\int_0^\alpha \| 2(W+V)e^{-t \Delta} f \|_1\, dt \le \gamma \| f \|_1.
\end{equation} 
In this case, $2(W+V)$ is called a {\it Miyadera perturbation of $\Delta$}. This estimate implies that $\Delta + 2(W+V)$ is the generator of a strongly continuous semigroup on $L^1(M)$. The class of Miyadera perturbations is larger than the well known Kato class. For all this we refer to \cite{Voigt}. 
\begin{proposition}\label{propKato2}
Suppose \eqref{eq-Miyadera}. Given  $\epsilon \in (0, \frac{1}{2}]$ and $\nu > 0$. Then for all $p \in (2, \infty)$ and 
 $p_0, p_1 \in (1, \infty)$ such that $\frac{1}{p_0} + \frac{1}{p_1} = \frac{2}{p}$, 
\begin{equation*}
\| \nabla f \|_p^2 \le C_\epsilon \| (\Delta + \nu)^{\frac{1}{2} + \epsilon} f \|_{p_0}\| (\Delta + \nu)^{\frac{1}{2} - \epsilon} f \|_{p_1}
\end{equation*}
for all $f \in D(\Delta^{\frac{1}{2} + \epsilon}) \cap D(\Delta^{\frac{1}{2} - \epsilon})$.
\end{proposition} 

\begin{proof}
We argue as in the proof of the previous proposition. Since $W+V$ and $2(W+V)$ are Miyadera  pertubations of $\Delta$, both $\Delta + W + V$ and $\Delta + 2(W+V)$ are generators of semigroups on $L^1(M)$. In particular, there exists a contants $a \ge 0$ such that $e^{-t(\Delta + 2(W+V) + a)}$ is uniformly bounded on $L^\infty(M)$. In addition, the domination \eqref{eq-domination} obviously holds with the additional term $e^{-t a}$ on both sides. We can now argue  exactly as in the previous proof to obtain the gradient estimate
$$ | \nabla e^{-t (\Delta + a)} f |^2 \le c\, e^{-t(\Delta + a)} |\nabla f |^2.$$
Arguing again as before with $\Delta + a$ instead of $\Delta$, the later inequality implies the boundedness of the Littlewood-Paley-Stein functional associated with $\Delta + a$ on $L^p(M)$ and hence
\begin{equation*}
\| \nabla f \|_p^2 \le C_\epsilon \| (\Delta + a)^{\frac{1}{2} + \epsilon} f \|_{p_0}\| (\Delta + a)^{\frac{1}{2} - \epsilon} f \|_{p_1}.
\end{equation*}
For $\nu > 0$, notice that $\Delta + a = (\Delta + a)(\Delta + \nu)^{-1} (\Delta + \nu)$ and $(\Delta + a)(\Delta + \nu)^{-1}$ is a bounded operator. This allows to bound the norms of $(\Delta + a)^{\frac{1}{2} \pm \epsilon} f$ by those of $(\Delta + \nu)^{\frac{1}{2} \pm \epsilon} f$ and we obtain  the proposition. 
\end{proof} 

Some similar  arguments to the previous ones can be found in \cite{XDLi}, Section 2. The assumptions there are more restrictive but  the results give boundedness of the (global or local) Riesz transform.

%%%%%%%%%%%%%%%%%%%%%%%%%%%%%%%%%%%%%%%%%%%

\section{Littlewood-Paley-Stein functionals on differential forms}\label{sec4}

Let $M$ be  again a complete Riemannian manifold and consider as before  the Hodge-de Rham Laplacian  $ \vec{\Delta} = {\bf d} {\bf d}^* + {\bf d}^* {\bf d}$
on $L^2$ of $1$-differential forms over $M$. 
Recall the commutation property 
$$ {\bf d}^* \vec{\Delta} = \Delta {\bf d}^* \quad {\rm and} \quad  {\bf d} \Delta = \vec{\Delta}{\bf d}.$$
The first equality is on forms and the second one is on functions. 
It follows from this that 
$$ ({\bf d} e^{-t\Delta})^* =  e^{-t{\Delta}}{\bf d}^* =  {\bf d}^* e^{-t\vec{\Delta}}$$
and hence  the Littlewood-Paly-Stein  functional 
$$ \cev{H} \omega = \left( \int_0^\infty | {\bf d}^* e^{-t\vec{\Delta}} \omega |^2 \, dt \right)^{1/2}$$
on differential forms $\omega  \in L^2$ can be seen as a "dual" of the functional $H_{{\bf d}}$ on functions where
$$ H_{{\bf d}} f = \left( \int_0^\infty | {\bf d} e^{-t \Delta} f |^2 \, dt \right)^{1/2}.$$
The functional $H_{{\bf d}}$ is the same as $H_{\nabla}$ where we have replaced the gradient by the differential ${\bf d}$. Since 
$H_{{\bf d}}$ is always bounded on $L^p(M)$ for $p \in (1, 2]$ by Theorem \ref{th21}, one would expect $\cev{H}$ to be bounded on $L^q$ for all
$q \in [2, \infty)$. The biggest advantage of $\cev{H}$ is its clear relationship to the Riesz transform as shown in our next result. 
%In order to include local Riesz transforms ${\bf d} (\Delta + \kappa)^{-\frac{1}{2}}$ for some constant $\kappa$ we replace $\vec{\Delta}$  by 
%$\vec{\Delta} + \kappa$ in the definition of $\cev{H}$ and set
%$$ \cev{H}_\kappa \omega = \left( \int_0^\infty | {\bf d}^* e^{-t(\vec{\Delta} + \kappa)} \omega |^2 \, dt \right)^{1/2}.$$
Let us recall the {\it reverse Riesz inequality} on $L^q(M)$
\begin{equation}\label{eq4kappa}
% \| (\Delta + \kappa)^{1/2} f \|_q \le C_\kappa ( \| \nabla f \|_q + \| f \|_q)
\| \Delta^{1/2} f \|_q \le C \| \nabla f \|_q
\end{equation}
for some constant $C$ and all $f \in W^{1,q}(M)$. A standard  duality argument shows  that the boundedness of ${\bf d} \Delta^{-\frac{1}{2}}$ on $L^p(M)$ implies the reverse Riesz inequality 
\eqref{eq4kappa} on the dual space $L^{p'}(M)$.  

\begin{theorem}\label{thm411} Let $p \in (1, \infty)$.  %Let $\kappa \ge 0$ be a constant and consider the  following assertions 
Consider the following assertions
\begin{enumerate}
\item The Riesz transform ${\bf d}\Delta^{-\frac{1}{2}}$ is bounded on $L^p(M)$
\item $ \cev{H}$ is bounded on $L^{p'}$ of $1$-differential forms (with values in $L^{p'}(M)$)
\item The reverse Riesz inequality \eqref{eq4kappa} holds on $L^{p'}(M)$ 
\item $\cev{H}$ is bounded on $L^{p'}$ of exact forms (i.e., $\omega = {\bf d} f \in L^{p'}$).
\end{enumerate}
Then assertions $1)$ and  $ 2)$ are equivalent,  and $3)$ and $ 4)$ are equivalent. 
\end{theorem} 
In order to prove the theorem we start with the following lemma. 
\begin{lemma} Let $\kappa \ge 0$. Then 
${\bf d} (\Delta + \kappa)^{-\frac{1}{2}}$  is bounded on $L^p$ {\it if and only if}
$$ \| \int_0^\infty {\bf d} e^{-t(\Delta+ \kappa)} F(t) \, dt \|_p \le C \| F \|_{L^p(L^2_t)} := C \| \left( \int_0^\infty | F(t) |^2\, dt \right)^{1/2} \|_p$$
for some constant $C$ and all $F \in L^p(L^2_t)$. 
\end{lemma}

\begin{proof} We prove the direct implication. 
We have
  \begin{eqnarray*}
  \| \int_0^\infty {\bf d} e^{-t(\Delta +\kappa)} F(t) \, dt \|_p &= &
 \| {\bf d} (\Delta + \kappa)^{-\frac{1}{2}} \int_0^\infty  (\Delta + \kappa)^{\frac{1}{2}} e^{-t(\Delta + \kappa)} F(t) \, dt \|_p\\
 &\le& C_0  \| \int_0^\infty  (\Delta + \kappa)^{\frac{1}{2}} e^{-t(\Delta+\kappa)} F(t) \, dt \|_p.
 \end{eqnarray*}
In order to estimate the latest term we argue by duality.  Let $g \in L^{p'}(M)$. Then
\begin{eqnarray*} 
&&\hspace{-2cm}  | \langle \int_0^\infty (\Delta + \kappa)^{\frac{1}{2}} e^{-t(\Delta+ \kappa)} F(t) \, dt, g \rangle | \\
&=& |   \int_M \int_0^\infty  F(t)  (\Delta + \kappa)^{\frac{1}{2}} e^{-t(\Delta+ \kappa)}g\, dt\, dx |\\
&\le& \int_M \left( \int_0^\infty | F(t) |^2\, dt \right)^{1/2}   \left( \int_0^\infty | (\Delta + \kappa)^{\frac{1}{2}} e^{-t(\Delta+ \kappa)} g |^2\, dt \right)^{1/2}\, dx\\
&\le& C_1 \| F \|_{L^p(L^2_t)} \|g \|_{p'}
\end{eqnarray*}
where we used the square function estimate for $\Delta + \kappa$ on $L^{p'}$. Note that $e^{-t(\Delta+ \kappa)}$ is a sub-Markovian semigroup and hence the generator has bounded square functions on $L^r(M)$ for all $r \in (1, \infty)$. This fact was already mentioned and used several times in the previous sections. 

For the converse, let $f \in L^p(M)$ and  take $F(t) = (\Delta + \kappa)^{\frac{1}{2}} e^{-t(\Delta+ \kappa)} f$. Again, by the square function estimate, $F \in L^p(L^2_t)$. Now, notice that 
 $\int_0^\infty {\bf d} e^{-t(\Delta+ \kappa)}  F(t) \, dt$ coincides  (up to a positive constant) with ${\bf d} (\Delta + \kappa)^{-\frac{1}{2}}$. 
\end{proof}

\begin{proof}[Proof of Theorem \ref{thm411}]
For $1) \Leftrightarrow 2)$, we apply  the previous lemma (with $\kappa = 0$).  Thus,  assertion $1)$ is equivalent to the fact that the operator 
$$ S: F \mapsto \int_0^\infty {\bf d} e^{-t\Delta} F(t)\, dt$$
is bounded from $L^p(L^2_t)$ into $L^p$ (of $1$- forms). 
 This is equivalent to the boundedness from $L^{p'}$ into $L^{p'}(L^2_t)$ of  its adjoint $S^*$. Now,  $S^* (\omega) = {\bf d}^* e^{-t\vec{\Delta}} \omega$ and we obtain  assertion  $2)$.

Next we prove  $3) \Leftrightarrow 4)$. We have 
\begin{eqnarray*} 
\| \left(\int_0^\infty  | {\bf d}^* e^{-t \vec{\Delta}} {\bf d} f |^2 \, dt \right)^{1/2} \|_{p'} &=& \| \left( \int_0^\infty | {\bf d}^* {\bf d} e^{-t {\Delta}}  f |^2 \, dt \right)^{1/2} \|_{p'}\\
&=& \| \left( \int_0^\infty | \Delta^{\frac{1}{2}} e^{-t {\Delta}}   \Delta^{\frac{1}{2}} f |^2 \, dt \right)^{1/2} \|_{p'}.
\end{eqnarray*}
By upper and lower square function estimates the latest term is equivalent to 
$\| \Delta^{\frac{1}{2}} f \|_{p'}$.   This  finishes the proof. 
\end{proof}

If one is interested in the local Riesz transform ${\bf d} (\Delta + \kappa)^{-1/2}$ for some $\kappa > 0$, then the above characterizations remain valid with the same  proofs.  More precisely,
${\bf d} (\Delta + \kappa)^{-1/2}$ is bounded on $L^p$ {\it if and only if} the functional
$$\cev{H}_\kappa \omega = \left( \int_0^\infty | {\bf d}^* e^{-t(\vec{\Delta} + \kappa)} \omega |^2 \, dt \right)^{1/2}$$
is bounded on $L^{p'}$ of $1$-forms. The reverse Riesz inequality 
$$ \| \Delta^{1/2} f \|_{p'} \le C ( \| \nabla f \|_{p'} + \| f \|_{p'})$$
holds {\it if and only if} 
$$ \| \cev{H}({\bf d}f) \|_{p'} \le C' ( \| \nabla f \|_{p'} + \| f \|_{p'}).$$
We mention that the implication  $2) \Rightarrow 1)$ in Theorem \ref{thm411} was proved  in \cite{Cometx}.

The  arguments in the proof of Theorem \ref{thm411} also  show that the boundedness of the Riesz transform 
can be characterized  by a lower bound for the Littlewood-Paley-Stein functional.
\begin{proposition}\label{prop411} Let $p \in (1, \infty)$.   Then the  Riesz transform ${\bf d} \Delta^{-1/2} $ is bounded on $L^p$  if and only if there exists a constant $C' > 0$ such that 
$$ \| \cev{H} \omega \|_p \ge C' \| \omega \|_p \quad  \forall \omega = {\bf d} f \in L^p.$$
\end{proposition} 
\begin{proof} We have seen in the proof of $3) \Leftrightarrow 4)$  of Theorem \ref{thm411} that 
$ \| \cev{H} ({\bf d}f) \|_p $ is equivalent to $ \| \Delta^{1/2} f \|_{p}$. This clearly implies the assertion of the proposition. 
\end{proof}

\begin{remark} Combining  Theorem \ref{thm411} and Proposition \ref{prop411} yields  that 
$$ \| \cev{H} \omega \|_{p'} \le C \| \omega \|_{p'} \quad   {\rm for\ all\ forms\ of\ order\ one}\ \omega \in L^{p'}$$
is equivalent to 
$$\| \cev{H} \omega \|_p \ge C' \| \omega \|_p \quad   {\rm for\ all\ exact\ forms } \ \omega = {\bf d} f \in L^p.$$
While it is easy to obtain the lower bound from the upper bound by a duality argument, it came as a surprise (at least to the author) that a  lower bound on exact forms implies an upper bound for all forms of order $1$ on the dual space.  
\end{remark} 

As mentioned in the introduction, there are many manifolds for which  it is known that the Riesz transform is bounded on $L^p$. For such manifolds,  the upper bound of $\cev{H}$ holds on $L^{p'}$ of $1$-forms and the lower bound holds on $L^p$ of exact forms. If $M$ is a Vicsek manifold then it is proved in \cite{Chen-Coulhon-Russ}, Theorem 5.3,  that the Riesz transform is bounded on $L^p$ {\it if and only if} $p \in (1,2]$ and  the reverse Riesz transform holds on $L^q$ {\it if and only if} $q \in [2, \infty)$. It follows from the previous results that for this manifold, $\cev{H}$ is bounded on $L^r$ of $1$-forms {\it if and only if} $r \in [2, \infty)$ and it is unbounded even for exact forms on $L^r$  for all $r \in (1, 2)$. The lower estimate of Proposition \ref{prop411} holds {\it if and only if} $r \in (1, 2]$.

%%%%%%%%%%%%%%%%%%%%%%%%%%%%%%%%%%%%%%%%%%%%%
\section{Appendix:  P.A. Meyer's inequality}\label{App1}

  P.A. Meyer \cite{Meyer}, \cite{Meyer2}  proved variations of Littlewood-Paley-Stein inequalities. One of them states  that the functional
  \begin{equation}\label{defS0}
  S_0(f) = \left( \int_0^\infty e^{-t \sqrt{\Delta}} | \nabla e^{-t\sqrt{\Delta}} f |^2\, tdt \right)^{1/2}
  \end{equation} 
  is bounded on $L^p(M)$ for all $p \in [2, \infty)$. 
 The lecture notes \cite{Meyer} and \cite{Meyer2} are  collections of talks given at the seminar of Probability at the university of Strasbourg. 
   They contain a lot of interesting  material. The setting there  is general but   the proofs and statements are 
   unfortunately buried in different notations which  
 sometimes change from one {\it S\'eminaire}   to another and that  makes the reading rather difficult.  The proof of the aforementioned  result on $S_0$ is heavily 
 based on probabilistic arguments. It would be very interesting to have an analytic and relatively short proof  
 but this seems to be  missing in the literature. In this Appendix, we give such a proof  when the gradient 
 estimate \eqref{eq38} is satisfied. This is the context of Section \ref{sec5} where we used P.A. Meyer's inequality. 
 We prove the  estimate  with the  heat semigroup $e^{-t\Delta}$ in the definition of $S_0$ and not the Poisson 
 one $e^{-t \sqrt{\Delta}}$ (see \eqref{defS} below). This is more general since one can deduce the estimates for $S_0$ 
 by arguments based on the subordination formula 
 $$ e^{-t\sqrt{\Delta}} = \int_0^\infty e^{-u} u^{-1/2} e^{- \frac{t^2}{4u} \Delta} \, du.$$
We leave the details to the reader and focus on the proof of the proposition below. 

Let $M$ be a complete Riemannian manifold. We assume in this section that the gradient estimate  \eqref{eq38} holds with some constants $\theta > 0$ and $c_{\theta} > 0$. Then  we have
\begin{proposition}\label{S-meyer}
Given two constants  $\alpha, \beta > 0 $ and define the functional 
\begin{equation}\label{defS}
 S (f) = \left(  \int_0^\infty e^{-t \alpha \Delta} | \nabla e^{-t\beta \Delta} f |^2\, dt \right)^{1/2}.
  \end{equation} 
Then $S$ is bounded on $L^p(M)$ for all $p \in [2, \infty)$.
\end{proposition} 

\begin{proof} First, observe that by the change of variable  $s= t \beta$ we can always assume that $\beta = 1$.\\
 Let  $ p \in [2, \infty)$. It is easy to see that the functional $S$ is quasi-linear and bounded on $L^2(M)$. Therefore, by interpolation, we may  assume in the sequel that $p > 4$. 

Let $0 < b < \infty$ and set 
$$I_{ b}(f) = \left( \int_0^b e^{-t \alpha \Delta} | \nabla e^{-t\Delta} f |^2\, dt \right)^{1/2}.$$
Then for $f \in C_c^\infty(M)$, 
\begin{eqnarray*}
\| I_{ b}(f) \|_p^2 &= & \| \int_0^b e^{-t \alpha\Delta} | \nabla e^{-t\Delta} f |^2\, dt \|_{p/2}\\
&\le& \int_0^b \| e^{-t\alpha \Delta} | \nabla e^{-t\Delta} f |^2 \|_{p/2}\, dt \\
&\le& \int_0^b \|   \nabla e^{-t\Delta} f \|_p^2\, dt \\
&\le& c_\theta \int_0^b \| e^{-t\theta \Delta} |\nabla f|^2 \|_{p/2}\, dt \\
& \le & c_\theta b\, \| \nabla f \|_p^2 < \infty.
\end{eqnarray*}
Of course, this estimate is not good enough but we want to have beforehand that $\| I_{ b}(f) \|_p$ is finite for fixed $0 < b < \infty$.  This will be used later in this proof. Next, let $q \in (1, 2)$ be such that $\frac{2}{p} + \frac{1}{q} = 1$. Let $0\le \varphi \in L^q(M)$. We argue by duality and thus we need to estimate
$$ J_{ b} = \int_M I_{b}(f)^2 \varphi \, dx,$$
where we denote again by $dx$ the integration with respect to the Riemannian measure on $M$. We have
\begin{eqnarray*}
J_{ b} &=&  \int_0^b  \int_M \nabla e^{-t\Delta} f. \nabla e^{-t\Delta} f. e^{-t \alpha\Delta} \varphi \, dx\, dt\\
&=& \int_0^b  \int_M  e^{-t\Delta} f. \Delta  e^{-t\Delta} f. e^{-t\alpha\Delta} \varphi \, dx\, dt - 
\int_0^b  \int_M  e^{-t\Delta} f. \nabla e^{-t\Delta} f. \nabla e^{-t\alpha\Delta} \varphi \, dx\, dt. 
\end{eqnarray*}
  Integrating by parts (first with respect to $t$ and then with respect to $x$) the first term in the above difference coincides with
  \begin{eqnarray*}
  &&  \frac{1}{2}\int_M (f^2 - ( e^{-t b \Delta}f)^2) \varphi\, dx - \frac{\alpha}{2} \int_M \int_0^b (e^{-t\Delta} f)^2 \Delta e^{-t\alpha\Delta} \varphi\, dt\, dx\\
  &=& \frac{1}{2}\int_M (f^2 - ( e^{-t b \Delta}f)^2) \varphi\, dx -  \alpha \int_M \int_0^b e^{-t\Delta} f. \nabla e^{-t\Delta} f. \nabla   e^{-t \alpha\Delta} \varphi\, dt\, dx.
  \end{eqnarray*}
  Hence
  \begin{equation}\label{eq51} 
  J_{ b} = \frac{1}{2}\int_M (f^2 - ( e^{-t b \Delta}f)^2) \varphi\, dx - (1+\alpha)  \int_M \int_0^b e^{-t\Delta} f. \nabla e^{-t\Delta} f. \nabla  e^{-t\alpha \Delta} \varphi\, dt\, dx.
  \end{equation} 
  Let $\delta \in (0,1)$  and write  $e^{-t\Delta} = e^{- \delta t\Delta} e^{-(1-\delta)t\Delta}$. Using the gradient estimate \eqref{eq38} it follows that 
  $$ | e^{-t\Delta} f. \nabla e^{-t\Delta} f. \nabla  e^{-t \alpha \Delta} \varphi | \le \sqrt{c_\theta} (\sup_{t \ge 0} e^{-t\Delta} |f|) ( e^{-t \delta \theta \Delta} | \nabla e^{- (1-\delta) t \Delta} f |^2 )^{1/2}
  |\nabla  e^{-t \alpha \Delta} \varphi |.$$
  Thus,
  \begin{eqnarray*}
  &&\hspace{-1cm} | \int_M \int_0^b e^{-t\Delta} f. \nabla e^{-t\Delta} f. \nabla  e^{-t \alpha \Delta} \varphi\, dt\, dx | \\
  &\le& \sqrt{c_\theta} \int_M \sup_{t \ge 0} e^{-t\Delta} |f| \int_0^b \left( e^{-t \delta \theta \Delta} | \nabla e^{- (1-\delta) t \Delta} f |^2 \right)^{1/2}
  |\nabla  e^{-t \alpha \Delta} \varphi |\, dt\, dx\\
  &=&  \sqrt{c_\theta} \int_M \sup_{t \ge 0} e^{-t\Delta} |f| \int_0^{b(1-\delta)}  \left( e^{-t \frac{\delta \theta}{1-\delta} \Delta} | \nabla e^{- t \Delta} f |^2 \right)^{1/2}
  |\nabla  e^{-t \frac{\alpha}{1-\delta} \Delta} \varphi |\, dt\, dx\\
  &\le & \sqrt{c_\theta} \int_M \sup_{t \ge 0} e^{-t\Delta} |f| \left( \int_0^{b(1-\delta)}  e^{-t \frac{\delta \theta}{1-\delta} \Delta} | \nabla e^{- t \Delta} f |^2\, dt \right)^{1/2} 
  \left(\int_0^\infty  | \nabla e^{-t \frac{\alpha}{1-\delta} \Delta} \varphi |^2\, dt \right)^{1/2}\, dx\\
  &\le& c_1 \| f \|_p \| \left( \int_0^{b}  e^{-t \frac{\delta \theta}{1-\delta} \Delta} | \nabla e^{- t \Delta} f |^2\, dt \right)^{1/2}  \|_p  \| \varphi \|_q
  \end{eqnarray*}
  for some constant $c_1$ independent of  $b$ since we only use the boundedness on $L^p(M)$ 
  of the maximal operator $f \mapsto \sup_{t\ge 0} e^{-t\Delta}|f|$  and the fact that the Littlewood-Paley-Stein functional $H_\nabla$ is bounded on $L^q(M)$ because  $ q \in (1, 2)$. 
  %Since $\delta \in (0, 1]$ it follows  from the obvious change of the variable $s = t \delta$ that
 % $$\int_\epsilon^b e^{-t \delta \Delta} | \nabla e^{-t \delta \Delta} f |^2\, dt \le \frac{1}{\delta} \int_\epsilon^b e^{-t \delta \Delta} | \nabla e^{-t \delta \Delta} f |^2\, dt.$$
  We chose $\delta = \frac{\alpha}{\theta+ \alpha}$ so that $e^{-t \frac{\delta \theta}{1-\delta} \Delta} = e^{-t \alpha \Delta}$ and hence we have proved that 
 % \begin{equation}\label{eq42}
  $$ | \int_M \int_0^b e^{-t\Delta} f. \nabla e^{-t\Delta} f. \nabla  e^{-t \alpha \Delta} \varphi\, dt\, dx | \le c_1 \| f \|_p \| \left(\int_0^b e^{-t \alpha \Delta} | \nabla e^{-t  \Delta} f |^2\, dt \right)^{1/2} \|_p  \| \varphi \|_q.$$
 %  \end{equation} 
 We insert this in \eqref{eq51} to obtain 
\begin{equation*}
\int_M I_{ b}(f)^2 \varphi \, dx  \le \| f \|_p^2 \| \varphi\|_q + (1+\alpha)c_1 \| f\|_p \| I_{b}(f) \|_p \|\varphi\|_q.
\end{equation*} 
Since this inequality is valid for all non-negative $\varphi \in L^q(M)$ and $\| I_{ b}(f) \|_p$ is finite, we obtain $I_{ b}(f)^2 \in L^{p/2}(M)$ and
$$ \| I_{ b}(f)^2 \|_{p/2} \le \| f \|_p^2 + (1+\alpha)c_1  \| f\|_p \| I_{b}(f) \|_p.$$
This obviously gives
$$ \| I_{ b}(f) \|_p \le c\, \|f\|_p$$
with a constant $c > 0$ independent of $b$. We let  $b \to + \infty$ to obtain
\begin{equation}\label{eq54}
 \| \left( \int_0^\infty e^{-t \alpha \Delta}| \nabla e^{-t\Delta} f |^2\, dt \right)^{1/2} \|_p \le c\,  \|f \|_p
 \end{equation}
for all $f \in C_c^\infty(M)$. By density arguments  this inequality extends to all $f \in L^p(M)$. 
A way to see this is to write for $f \in C_c^\infty(M)$
$$ | \nabla f |^2 = - \int_0^\infty \frac{d}{dt} | \nabla e^{-t\Delta} f |^2 = 2 \int_0^\infty \nabla e^{-t\Delta} \Delta f. \nabla e^{-t\Delta} f\, dt$$
from which one obtains as in the proof of Proposition \ref{prop3-333} (applied to $f$ and $\Delta f$) 
that for some constant $C > 0$
$$\| \nabla  f \|_p^2 \le C \| \Delta f \|_p \| f \|_p.$$
This extends to all $f $ in the domain of $\Delta$ as an operator on $L^p(M)$ and hence  
$$ \| \nabla e^{-t\Delta} f \|_p \le \frac{C}{\sqrt{t}} \|f \|_p$$
for all $f \in L^p(M)$. It follows from this that for $\epsilon > 0$ and $b < \infty$, if $(f_n)_n \in C_c^\infty(M)$ converges to 
$f$ in $L^p$, then $ \left( \int_\epsilon^b e^{-t\alpha \Delta} | \nabla e^{-t \Delta} f_n |^2\, dt \right)^{1/2}$ converges to 
$ \left( \int_\epsilon^b e^{-t\alpha \Delta} | \nabla e^{-t \Delta} f |^2\, dt \right)^{1/2}$ in $L^p$. Now  \eqref{eq54} implies 
$$ \| \left( \int_\epsilon^b e^{-t \alpha \Delta}| \nabla e^{-t\Delta} f_n |^2\, dt \right)^{1/2} \|_p \le c\,  \|f_n \|_p$$
with constant $c$ independent of $n, \epsilon$ and $b$. We let $n \to \infty$ and then $\epsilon \to 0, b \to \infty$ to obtain 
\eqref{eq54} for all $f \in L^p(M)$. 
 
\end{proof}

%%%%%%%%%%%%%%%%%%%%%%%%%
%%%%%%%%%%%%%%%%%%%%%%%%%

\end{document}